\documentclass[a4paper,leqno]{article}

\usepackage[usenames,dvipsnames]{xcolor}
\usepackage{eso-pic, picture, graphicx}
\usepackage{natbib}
\usepackage[shortlabels, inline]{enumitem}
\usepackage{authblk}
\usepackage{geometry}
\usepackage{marvosym}
\usepackage{tikz-cd}
\usetikzlibrary{fit,shapes.geometric}
\usepackage{type1ec}
\usepackage{verbatim}
\usepackage[T1]{fontenc}
\usepackage{lettrine}
\usepackage{amssymb}
\usepackage{amsmath}
\usepackage{amsthm}
\usepackage{mathrsfs}
\usepackage{setspace}
\usepackage{leqno}
\usepackage{rotating}
\usepackage{qtree}
\usepackage[retainorgcmds]{IEEEtrantools}
\usepackage[utf8]{inputenc}
\usepackage{afterpage}
\usepackage{newlfont}
\usepackage{tikz}
\usepackage{xcolor}
\usepackage{xspace}
\usepackage{thm-restate} 
\usetikzlibrary{arrows}
\usepackage{hyperref}
\hypersetup{
			colorlinks=true,
			linkcolor=blue,
			citecolor=blue,
}
\usepackage{mathdots}
\usepackage{marginnote}
\usepackage{bussproofs-extra}
\def\fCenter{{\mbox{$\hspace{0.5mm}\GD$}}}
\newcommand{\Ax}{\Axiom}
\newcommand{\RL}[1]{\RightLabel{\footnotesize{#1}}}
\newcommand{\Un}{\UnaryInf}
\newcommand{\Bn}{\BinaryInf}
\newcommand{\Dis}{\DisplayProof}
\newcommand{\AxC}{\AxiomC}
\newcommand{\UnC}{\UnaryInfC}
\newcommand{\BnC}{\BinaryInfC}

\setcitestyle{authoryear,open={(},close={)}}

\newcommand{\mypar}[1]{\medskip\noindent\textbf{#1.}}

\newcommand{\PA}{\textnormal{\textsf{PA}}\xspace}
\newcommand{\KF}{\textnormal{\textsf{KF}}\xspace}
\newcommand{\KFcs}{\ensuremath{\textnormal{\textsf{KF}}_\textnormal{\textsf{cs}}}\xspace}
\newcommand{\KFcp}{\ensuremath{\textnormal{\textsf{KF}}_\textnormal{\textsf{cp}}}\xspace}
\newcommand{\KFs}{\ensuremath{\textnormal{\textsf{KF}}_\textnormal{\textsf{S}}}\xspace}
\newcommand{\KFst}{\ensuremath{\KF_\star}\xspace}
\newcommand{\KFint}{\ensuremath{\textnormal{\KF}^\textnormal{\textsf{int}}}\xspace}
\newcommand{\KFintcs}{\ensuremath{\textnormal{\textsf{KF}}^\textnormal{\sf int}_\textnormal{\textsf{cs}}}\xspace}
\newcommand{\KFintcp}{\ensuremath{\textnormal{\textsf{KF}}^\textnormal{\sf int}_\textnormal{\textsf{cp}}}\xspace}
\newcommand{\KFints}{\ensuremath{\textnormal{\textsf{KF}}^\textnormal{\sf int}_\textnormal{\textsf{S}}}\xspace}
\newcommand{\KFintst}{\ensuremath{\KFint_\star}\xspace}

\newcommand{\PKF}{\textnormal{\textsf{PKF}}\xspace}
\newcommand{\PKFcs}{\ensuremath{\textnormal{\textsf{PKF}}_\textnormal{\textsf{cs}}}\xspace}
\newcommand{\PKFcp}{\ensuremath{\textnormal{\textsf{PKF}}_\textnormal{\textsf{cp}}}\xspace}
\newcommand{\PKFs}{\ensuremath{\textnormal{\textsf{PKF}}_\textnormal{\textsf{S}}}\xspace}
\newcommand{\PKFst}{\ensuremath{\PKF_\star}\xspace}
\newcommand{\FDE}{\textnormal{\textsf{FDE}}\xspace}
\newcommand{\K}{\ensuremath{\textnormal{\textsf{K}}\bo{3}}\xspace}
\newcommand{\LP}{\textnormal{\textsf{LP}}\xspace}
\newcommand{\KS}{\ensuremath{\textnormal{\textsf{KS}}\bo{3}}\xspace}
\newcommand{\CL}{\textnormal{\textsf{CL}}\xspace}

\newcommand{\then}{\rightarrow}
\newcommand{\sse}{\leftrightarrow}
\newcommand{\Then}{\Rightarrow}

\newcommand{\Sub}{\subseteq}
\newcommand{\cor}[1]{\ulcorner{#1}\urcorner}
\newcommand{\ang}[1]{\langle{#1}\rangle}
\newcommand{\ert}[1]{\lvert{#1}\rvert}

\DeclareMathOperator*{\mvee}{\text{\raisebox{0.25ex}{\scalebox{0.7}{$\bigvee$}}}}
\DeclareMathOperator*{\mwedge}{\text{\raisebox{0.25ex}{\scalebox{0.7}{$\bigwedge$}}}}

\newcommand{\lpa}{\ensuremath{\mathcal{L}_{\PA}}}
\newcommand{\lpat}{\ensuremath{\mathcal{L}_{T}}\xspace}
\newcommand{\Sent}{\textnormal{\textsf{Sent}}}

\newcommand{\Fml}{\textnormal{\textsf{Fml}}}
\newcommand{\Term}{\textnormal{\textsf{Term}}}
\newcommand{\ClTerm}{\textnormal{\textsf{ClTerm}}}
\newcommand{\subst}{\textnormal{\textsf{sub}}}
\newcommand{\num}{\textnormal{\textsf{num}}}
\newcommand{\Var}{\textnormal{\textsf{Var}}}
\newcommand{\val}{\textnormal{\textsf{val}}}
\newcommand{\Bew}{\textnormal{\textsf{Bew}}}

\newcommand{\negat}{\ensuremath{\textnormal{{\sf At}}^-}}

\newcommand{\har}{\ensuremath{\negthickspace\upharpoonright}\xspace}
\newcommand{\ud}[1]{\d{$#1$}}
\newcommand{\bo}[1]{\mathsf{#1}}

\newcommand{\GG}{\textnormal{\textsf{GG}}\xspace}

\newcommand{\Cons}{\textnormal{\textsf{Cons}}\xspace}
\newcommand{\Comp}{\textnormal{\textsf{Comp}}\xspace}

\newcommand{\Dmc}{\ensuremath{\mathcal{D}}\xspace}

\newcommand{\Rmc}{\ensuremath{\mathcal{R}}\xspace}

\newcommand{\Isf}{\textnormal{\textsf{I}}\xspace}

\newcommand{\Ssf}{\textnormal{\textsf{S}}\xspace}

\newcommand{\GD}{\ensuremath{\Gamma\Then\Delta}\xspace}
\newcommand{\GDv}{\ensuremath{\Gamma(\vec{x})\Then\Delta(\vec{y})}\xspace}
\newcommand{\GDp}{\ensuremath{\Gamma\Then\Delta'}\xspace}
\newcommand{\GpD}{\ensuremath{\Gamma'\Then\Delta}\xspace}
\newcommand{\GpDp}{\ensuremath{\Gamma'\Then\Delta'}\xspace}

\newcommand{\TGD}{\ensuremath{T\cor{\Gamma} \Then T\cor{\Delta}}\xspace}

\newcommand{\TGDv}{\ensuremath{T\cor{\Gamma(\dot{\vec{x}})} \Then T\cor{\Delta(\dot{\vec{y}})}}}

\newcommand{\TI}{\ensuremath{\textnormal{\textsf{TI}}}\xspace}

\makeatletter
\newcommand*\bigcdot{\mathpalette\bigcdot@{.8}}
\newcommand*\bigcdot@[2]{\mathbin{\vcenter{\hbox{\scalebox{#2}{$\m@th#1\bullet$}}}}}
\makeatother
\makeatletter
\newcommand\ztag[1]{%
	\def\@currentlabel{#1}%
	\gdef\tmp{%
	\addtocounter{equation}{-1}%
	\def\theequation{#1}}%
	\aftergroup\aftergroup\aftergroup\aftergroup\aftergroup\aftergroup
	\aftergroup\aftergroup\aftergroup\aftergroup\aftergroup\aftergroup
	\aftergroup\aftergroup\aftergroup\aftergroup\aftergroup\aftergroup
	\aftergroup\aftergroup\aftergroup\aftergroup\aftergroup\aftergroup
	\aftergroup\aftergroup\aftergroup\aftergroup\aftergroup\aftergroup
	\aftergroup
	\tmp
}

\theoremstyle{plain}
\newtheorem{definition}{Definition}[section]

\theoremstyle{plain}
\newtheorem{lemma}[definition]{Lemma}

\theoremstyle{plain}

\theoremstyle{plain}
\newtheorem{corollary}[definition]{Corollary}

\theoremstyle{plain}
\newtheorem{proposition}[definition]{Proposition}

\theoremstyle{plain}
\newtheorem{fact}[definition]{Fact}

\theoremstyle{plain}
\newtheorem{remark}[definition]{Remark}

\theoremstyle{remark}

\theoremstyle{plain}

\theoremstyle{definition}

\title{\KF, \PKF, and Reinhardt's Program}
\author{Luca Castaldo\thanks{luca.castaldo@bristol.ac.uk} \& Johannes Stern\thanks{johannes.stern@bristol.ac.uk}\\University of Bristol}

\date{}	
\begin{document}

\maketitle              
%

\begin{abstract}
\label{abstract}

In ``Some Remarks on Extending an Interpreting Theories with a Partial Truth Predicate'' \cite{reiRemarks} famously proposed an instrumentalist interpretation of the truth theory Kripke-Feferman (\KF) in analogy to Hilbert's program. Reinhardt suggested to view \KF as a tool for generating ``the significant part of \KF'', that is, as a tool for deriving sentences of the form $T\cor\varphi$. The constitutive question of \textit{Reinhardt's program} was whether it was possible ``to justify the use of nonsignificant sentences entirely within the framework of significant sentences''? This question was answered negatively by \cite{hahAxiomatizing} but we argue that under a more careful interpretation the question may receive a positive answer. To this end, we propose to shift attention from \KF-provably true \textit{sentences} to \KF-provably true \textit{inferences}, that is, we shall identify the significant part of \KF with the set of pairs $\ang{\Gamma, \Delta}$, such that \KF proves that if all members of $\Gamma$ are true, at least one member of $\Delta$ is true. In way of addressing Reinhardt's question we show that the provably true inferences of \KF coincide with the provable sequents of the theory Partial Kripke-Feferman (\PKF).

\end{abstract}

\allowdisplaybreaks
\section{Introduction}
\label{sec:Introduction}
	Kripke's theory of truth \citep{kriOutline} is a cornerstone of contemporary research on truth and the semantic paradoxes. The theory provides us with a strategy for constructing, that is defining, desirable interpretations of a self-applicable truth predicate, so-called fixed points. These fixed points can serve as interpretations of the truth predicate within non-classical models of the language, as in Kripke's original article, but can also be used in combination with classical models, so called closed-off models.\footnote{This was also suggested by Kripke \citep[cf.][p.715]{kriOutline}.} 
\cite{fefReflecting} devised an elegant axiomatic theory of the Kripkean truth predicate of these closed-off fixed-point models. The theory is known as Kripke-Feferman (\KF) and is still one of the most popular classical axiomatic truth theories in the literature. Nonetheless, \KF displays a number of unintended and slightly bizarre features, which it inherits from the behavior of the truth predicate in the closed-off fixed-point models. While in the non-classical fixed-point models the truth predicate is transparent, i.e., $\varphi$ and $T\cor\varphi$ will always receive the same semantic value, this no longer holds in the closed-off models. Rather for each closed-off model there will be sentences $\varphi$, e.g.~the Liar sentence, such that either $\varphi$ and $\neg T\cor\varphi$ will be true in the model, or $\neg\varphi$ and $T\cor\varphi$  will be true in the model. As a consequence one can prove this counterintuitive disjunction in \KF for the Liar sentence $\lambda$, i.e.
	\begin{flalign*}
		&(\ast)	&&	\mathsf{KF}\vdash(\lambda\wedge\neg T\cor\lambda)\vee(\neg\lambda\wedge T\cor\lambda).	&&
	\end{flalign*}
Since the transparency of truth seems to be one of the basic characteristica of the truth predicate, the aforementioned asymmetry puts the idea of understanding the closed-off models as suitable models of an intuitively acceptable truth predicate under some stress and alongside casts doubt on \KF as an acceptable theory of truth. However, reasoning within the non-classical logic of the Kripkean fixed-points seems a non-trivial affaire or, as Feferman would have it, ``{\it nothing like sustained ordinary reasoning can be carried on}'' in these non-classical logics \citep[p.~95]{fefToward}. Giving up on \KF thus hardly seems a desirable conclusion.

In reaction to the counterintuitive consequences of \KF, \cite{reiRemarks85, reiRemarks} proposed an instrumentalist interpretation of the theory in analogy to Hilbert's program. Famously, Hilbert proposed to justify number theory, analysis and even richer mathematical theories by finitary means. Without entering into Hilbert-exegesis, the main idea was of course to provide consistency proofs for these mathematical theories in a finitistically acceptable metatheory. From a finitist perspective this would turn the mathematical theories into useful tools for producing mathematical truths. But the fate of Hilbert's program, at least on its standard interpretation, is well known: G\"odel's incompleteness theorems are commonly thought to be the program's coffin nail.  Nonetheless \cite{reiRemarks85, reiRemarks} was optimistic that his program had greater chances of success.\footnote{On p.~225 \cite{reiRemarks} writes:
\begin{quote}``I would like to suggest that the chances of success in this context, where the interpreted or significant part of the language includes such powerful notations as truth, are somewhat better than in Hilbert's context, where the contentual part was very restictred.''\end{quote}} Reinhardt proposed to view \KF as a tool for deriving Kripkean truths in the same way Hilbert viewed, say, number theory as a tool for deriving mathematical truths. A Kripkean truth is a sentence that is true from the perspective of the Kripkean fixed-points: if $\mathsf{KF}\vdash T\cor\varphi$ ($T\cor{\neg\varphi}$), then $\varphi$ is true (false) in all non-classical fixed-point models, that is, we are guaranteed that $\varphi$ receives a semantic value from the perspective of Kripke's theory of truth. This is not a general feature of the theorems of \KF but peculiar to those sentences that \KF proves true (false). The latter sentences Reinhardt called ``{\it the significant part of} \KF'' \citep[p.~219]{reiRemarks} and labelled the set of \KF-significant sentences $\mathsf{KFS}:=\{\varphi\mid\mathsf{KF}\vdash T\cor\varphi\}$.\footnote{$\mathsf{KFS}$ is sometimes also called the {\it inner logic} of KF \citep[cf.][p.638]{hahAxiomatizing}. We decided to stick with Reimhardt's original terminology.} In light of this terminology the constitutive question of Reinhardt's program is whether it was possible ``{\it to justify the use of nonsignificant sentences entirely within the framework of significant sentences}'' \citep[p.~225]{reiRemarks}?

But what would such a successful instrumentalist interpretation of \KF and the use of non-significant sentences amount to? Reinhardt himself is scarce on the exact details, however, at the end of \cite{reiRemarks85} he asks the following question:
	\begin{quote}
		If $\mathsf{KF}\vdash T\cor\varphi$ is there a \KF-proof $$\varphi_1,\ldots,\varphi_n,T\cor\varphi$$ such that for each $1\leq i\leq n$, $\mathsf{KF}\vdash T\cor{\varphi_i}$. \citep[cf.~][p.~239]{reiRemarks85}\footnote{%
To be precise, \cite{reiRemarks85} asked whether there was a \KF-proof $\varphi_1,\ldots,\varphi_n,\varphi$ rather than a \KF-proof $\varphi_1,\ldots,\varphi_n,T\cor\varphi$. But this presupposes that all \KF-provably true sentences are also theorems of \KF. While this is true in the variants of \KF Reinhardt considers, this is not the case in all versions of \KF discussed in the literature. However, all remarks concerning our version of the question generalize to Reinhardt's original question.
}
	\end{quote} 
If we can answer this question positively it seems that we can justify each Kripkean truth provable in \KF by appealing solely to the significant fragment of \KF: even though we have reasoned in \KF, each step of our reasoning is part of {\sf KFS} and  hence ``remains within the framework of significant sentences''. This interpretation of Reinhardt's program is adopted by \cite{hahAxiomatizing} who called the question {\sc Reinhardt's Problem}. Unfortunately, as \cite{hahAxiomatizing} convincingly argue, if understood in this way the instrumentalist interpretation of \KF will fail. We refer to \cite{hahAxiomatizing} for details but, in a nutshell, the reason for this failure is that the truth-theoretic axioms of \KF will not be true in the non-classical fixed-point models and hence not be part of {\sf KFS}, e.g., if $\varphi_i:=\forall x ( \Sent(x) \then (T(x) \sse T(\ud\neg\ud\neg x) ) )$,\footnote{%
	See \S\ref{sec:Preliminaries} for details on notation.
}
then $\varphi_i\not\in \mathsf{KFS}$. Indeed, Halbach and Horsten take this to show ``{\it that Reinhardt's analogue of Hilbert's program suffers the same fate as Hilbert's program}'' \citep[p.~684]{hahAxiomatizing}.

However, we think that this conclusion is premature and argue that, to the contrary, if suitably understood Reinhardt's program can be deemed successful. Our key point of contention is that \cite{hahAxiomatizing}, arguably following Reinhardt, employ the perspective of classical logic 
when theorizing about the significant part of \KF. But the logic of the significant part of \KF is not classical logic but the logic of the non-classical fixed-point models, that is, a non-classical logic. This observation has two interrelated consequences for Reinhardt's program. First, contra Reinhardt and, Halbach and Horsten we should not identify the significant part of \KF exclusively with the set of significant \textit{sentences}. Rather it also seems crucial to ask which {\it inferences} are admissible within the significant part of \KF. Of course, in classical logic this difference collapses due to the deduction theorem but not so in non-classical logics. For example, the three-valued logic \textit{Strong Kleene}, \K, has no logical truth, but many valid inferences---if, in this case, we were to focus only on the theorems of the logic, there would be no logic to discuss. Moreover, since the significant sentences can be retrieved from the significant inferences, that is the provably true inferences, we should focus on the latter rather than the former in addressing {\sc Reinhardt's Problem}.\footnote{%
	The provable true sentences can be viewed as inferences where the truth of the sentence follows from an empty hypothesis.
}
To this end, it is helpful to conceive of \KF as formulated in a two-sided sequent calculus rather than a Hilbert-style axiomatic system. Let $\Gamma,\Delta$ be finite sets of sentences and let $T\cor\Gamma$ be short for $\{T\cor\gamma\mid\gamma\in\Gamma\}$. The admissible inferences of the significant part of \KF, which we label {\sf KFSI}, can then be defined as follows:\footnote{%
	Notice that moving to a two-sided sequent formulation of \KF is not essential. Due to the deduction theorem we can also define {\sf KFSI} be appeal to \KF formulated in an axiomatic Hilbert-style calculus. In this case, the definition would amount to 
	$$\mathsf{KFSI}:=\{\left\langle\Gamma,\Delta\right\rangle\mid\mathsf{KF}\vdash \mwedge T\cor\Gamma\rightarrow \mvee T\cor\Delta\}.$$
}
	\begin{align*}
		\mathsf{KFSI}:=\{\left\langle\Gamma,\Delta\right\rangle\mid\mathsf{KF}\vdash T\cor\Gamma\Rightarrow T\cor\Delta\}
	\end{align*}
 
Second, {\sc Reinhardt's Problem}, according to the formulation of \cite{hahAxiomatizing}, which admittedly was inspired by Reinhardt's \citeyearpar{reiRemarks85} original question, conceives of {\sf KF}-proofs as sequences of theorems of \KF. But by focusing on sequences of theorems, we cannot fully exploit the significant part of {\sf KF}, that is, {\sf KFSI} for precisely the reasons \cite{hahAxiomatizing} used to rebut Reinhardt's program: while double negation introduction is clearly a member of {\sf KFSI}, proving this fact by a sequence of theorems would take us outside of {\sf KFS} since it would use the truth-theoretic axiom $\forall x ( \Sent(x) \then (T(x) \sse T(\ud\neg\ud\neg x) ) )$, which is not a member of the significant part of {\sf KF}. This suggest a reformulation of {\sc Reinhard's Problem} in terms of a notion of proof that focuses on inferences rather than theorems. To this end, it proves useful again to formulate \KF in a two-sided sequent calculus and to conceive of proofs as derivation trees, where each node of the tree is labeled by a sequent. As a matter of fact in this case we can distinguish between two versions {\sc Reinhard's Problem}:

	\begin{enumerate}
		\item For every \KF-theorem of the form $T\cor\varphi$, is there a \KF-derivation tree with root $\emptyset\Rightarrow T\cor\varphi$ such that for every node $d$ of the tree, $d\in\mathsf{KFSI}$?
		\item For every \KF-derivable sequent of the form $T\cor\Gamma\Rightarrow T\cor\Delta$, is there a \KF-derivation tree with root $T\cor\Gamma\Rightarrow T\cor\Delta$ such that for every node $d$ of the tree, $d\in\mathsf{KFSI}$?
	\end{enumerate}

The first question is a reformulation of Halbach and Horsten's \citeyearpar{hahAxiomatizing} {\sc Reinhardt's Problem}. The second question, which we label {\sc Generalized Reinhardt Problem}, asks whether all provably true inferences can be justified by appealing to the significant inferences only. Arguably, to deem Reinhardt's program successful we need to give an affirmative answer to the {\sc Generalized Reinhardt Problem}. Otherwise, a proof of $T\cor\varphi$ could still rely on inferences that, whilst part of {\sf KFSI}, cannot themselves be justified by appealing only to the significant inferences of \KF. Perhaps surprisingly, we shall show that an affirmative answer to the {\sc Generalized Reinhardt Problem} can be given. It follows that on this more careful formulation Reinhardt's program can be deemed successful.

Arguably, one may still take issue with this conclusion and argue that our answer to the {\sc Generalized Reinhardt Problem} is at best a partial completion of Reinhardt's program: what is still required is an independent axiomatization of the significant part of \KF, for this would prove \KF dispensable. We postpone a discussion of this view to the conclusion. Rather we will now take a fresh look at the question of an independent axiomatization, which Halbach and Horsten called {\sc Reinhardt's Challenge} \citep[p.~689]{hahAxiomatizing}. This will prove instrumental in answering the {\sc Generalized Reinhardt Problem}. \cite{reiRemarks85} asked for an independent axiomatization of the significant part of \KF. More precisely, \citep[p.~239]{reiRemarks85} asked:
	\begin{itemize}
		\item[a)]``Is there an axiomatization of $\{\sigma\mid\mathsf{KF}\vdash T\cor\sigma\}$ which is natural and formulated entirely within the domain of significant sentences,\ldots.''
		\item[b)]``Similarly for the relation $\Gamma\vdash_{S}\sigma$ defined by $\mathsf{KF}+\{T\cor\gamma\mid\gamma\in\Gamma\}\vdash T\cor\sigma$.''
	\end{itemize}
\cite{hahAxiomatizing} proposed their theory Partial Kripke-Feferman (\PKF) in way of answering to {\sc Reinhardt's Challenge}.  \PKF is formulated in a non-classical, two-sided sequent calculus and thus fits neatly with our observation that one should focus on the provably true inferences of \KF rather than the provably true sentences. Moreover, \PKF is arguably a natural axiomatization of Kripke's theory of truth. However, \cite{hahAxiomatizing} observed that there are sentences $\varphi$ such that $\mathsf{KF}\vdash T\cor\varphi$ but $\mathsf{PKF}\not\vdash\varphi$, which led them to conclude that {\sc Reinhardt's Challenge} cannot be met. The reason for this asymmetry is due to the difference in proof-theoretic strength of \KF and \PKF: while \KF proves transfinite induction for ordinals below $\varepsilon_0$, \PKF only proves transfinite induction for ordinals smaller than $\omega^\omega$. As a consequence, there will be arithmetical sentences that \KF proves true that \PKF cannot prove. The story does not end there however. First, as \cite{hanCosts} observe, the discrepancy between \KF and \PKF arises only if the rule of induction is extended beyond the arithmetical language, that is, if we restrict induction to the language of arithmetic---call the resulting theories $\mathsf{KF}\har$ and $\mathsf{PKF}\har$---then $\mathsf{KF\har}\, \vdash T\cor\varphi$, if and only if, $\mathsf{PKF\har}\,\vdash\varphi$. This highlights that the asymmetry between \KF and \PKF is not due to the truth-specific principles but the amount of induction that is assumed in the respective theories. 

Second, corroborating the latter observation, \cite{nicProvably} showed that the asymmetry between \KF and \PKF is indeed solely due to the amount of induction available within the respective theories: Nicolai shows that if transfinite induction up to $<\varepsilon_0$ is added axiomatically to \PKF---call this theory $\mathsf{PKF^+}$--- then $\mathsf{KF}\vdash T\cor\varphi$, if and only if, $\mathsf{PKF^+}\vdash\varphi$. Moreover, \cite{nicProvably} shows that independently of which version of induction is assumed in \KF there will be a suitable \PKF-style theory, which will have exactly the provably true sentences of the relevant \KF-style theory as theorems. Nicolai took these results to ``{\it partially [accomplish] a variant of a program sketched by Reinhardt}'' \citep[cf.~][p.~103]{nicProvably}.

However, the work by Halbach and Nicolai provides at best a positive answer to Question a), it does not---at least not immediately---yield an answer to Question b). Indeed, it seems as if, despite working in a non-classical, two-sided sequent calculus \cite{hahAxiomatizing} have largely neglected Question b) of {\sc Reinhardt's Challenge} and so have subsequent publications on this issue. Indeed the version of \PKF originally proposed by \cite{hahAxiomatizing} failed to yield a positive answer to Question b) for rather banal reasons---even for theories with restricted induction: the version of \KF \cite{hahAxiomatizing} consider assumes the truth predicate to be consistent, which, as we shall explain in due course, means that the logic of {\sf KFSI} is \K. But Halbach and Horsten formulate \PKF in symmetric strong Kleene logic \KS and as a consequence $\mathsf{KF}\vdash T\cor\varphi,T\cor{\neg\varphi}\Rightarrow T\cor\psi$ while $\mathsf{PKF}\not\vdash\varphi,\neg\varphi\Rightarrow\psi$. The main technical contribution of this paper is to clarify the situation and to show how, building on Nicolai's \citeyearpar{nicProvably} work, a positive answer to Question b) of {\sc Reinhardt's Challenge} can be provided. To this end, we show how to pair the different variants of \KF with a suitable \PKF-style theory such that the provable sequents of the latter theory constitute exactly the significant inferences of the former theory. Moreover, it turns out that once we have an independent axiomatization of {\sf KFSI}, the {\sc Generalized Reinhardt Problem} can be answered rather immediately: it is easy to show that if a sequent $\Gamma\Rightarrow\Delta\in\mathsf{KFSI}$, then there will be a \KF-derivation tree such that each node of the tree can be derived in $\mathsf{PKF^+}$. But the provable sequents of $\mathsf{PKF^+}$ constitute exactly the significant inferences of \KF, that is, {\sf KFSI}---every node of the \KF-derivation is a member of {\sf KFSI}. The derivation remains within the significant part of \KF.

\subsection{Plan of the paper}
The paper starts by fixing some basic terminology and notation. More specifically, Section \ref{sec:Preliminaries} introduces the language and the logical systems underlying \PKF and its variants. That is, we introduce the logics \FDE, \KS, \K and \LP. In the next section, Section \ref{sec:KFandPKF}, we introduce the relevant families of \KF- and \PKF-like theories and observe some basic properties of these \PKF-systems. In Section \ref{sec:InternalLogic} we prove the central technical results of this paper. We show that for each \KF-like theory we can find a \PKF-counterpart such that the latter is an independent axiomatization of the significant inference of the former. In other words, we show that the set of pairs $\ang{\Gamma, \Delta}$ such that \TGD is derivable in a \KF-like theory coincides with the set of pairs $\ang{\Gamma, \Delta}$ such that \GD is derivable in a corresponding \PKF-like theory. It turns out a positive answer to {\sc Generalized Reinhardt Problem} is but an immediate corollary of the existence of such an independent axiomatization.


\section{Preliminaries}
\label{sec:Preliminaries}
\subsection{Language and notation} $\lpat$ denotes the language of first-order \emph{Peano arithmetic} (\PA) extended by a unary predicate $T$. $\lpa:=\lpat \backslash\{T\}$ is the $T$-free fragment of $\lpat$. Terms and formulae are generated in the usual way. By an \lpat-expression we mean a term or a formula of \lpat. $\bo n$ is the numeral corresponding to the number $n\in\omega$; $=$ is the formal identity symbol of \lpat. We fix a canonical \emph{G\"odel numbering} of $\lpat$-expressions. If $e$ is an $\lpat$-expression, the G\"odel number (= gn) of $e$ is denoted by $\#e$ and $\cor{e}$ is the term representing $\#e$ in $\lpa$. We introduce some primitive recursive relations, in practice working in a definitional extension of \PA:
		\begin{IEEEeqnarray*}{L}
			\Term(x) \ (\ClTerm(x)) := x \text{ is the gn of a (closed) term}; \\
			\Var(x) := x \text{ is the gn of variable}; \\
			\Fml_n(x) \ (\Sent(x)) := x \text{ is the gn of a formula with at most } n \ (0) \text{ free distinct variables}; \\
			\textsf{Eq}(x) := x \text{ is the gn of an equality between closed terms}; \\
			\textsf{Ver}(x) := x \text{ is the gn of true closed equality}.
		\end{IEEEeqnarray*}
We also assume to have standard representations of the following primitive recursive operations on G\"odel numbers:
		\begin{IEEEeqnarray*}{LL}
			\ud\land:\#\varphi,\#\psi \mapsto \#(\varphi\land\psi) \hspace{50mm} & \ud\neg:\#\varphi\mapsto\#(\neg\varphi)  \\
			\ud\forall:\#v_k,\#\varphi \mapsto \#(\forall v_k\varphi) & \ud=:\#t, \#s\mapsto\# (t= s) \\
			\ud T:\#t \mapsto \#T(t) \hspace{50mm} & \\
			\num:n \mapsto \#\bo n \ (\text{the gn of the } n \text{-th numeral}) & \subst:\#e,\#t,\#v_k \mapsto \#e[t/v_k]
		\end{IEEEeqnarray*}
$e[t/v_k]$ is the expression obtained from an expression $e$ by replacing each free occurrence of $v_k$ by a term $t$. We occasionally write $e(t)$, when it is clear which variable is being substituted. We also have a recursive function $\val(x)$ such that $\val(\cor t) = t$ for closed terms $t$. We write $\ud\exists a.b$ instead of $\ud\exists(a, b)$. We often abbreviate $\num(x)$ with $\dot{x}$. We write $\varphi(\vec{v})$ to denote a formula $\varphi$ whose free variables are contained in $\vec{v}$, and $\cor{\varphi(\dot{x})}$ for $\subst(\cor{\varphi(v)}, \num(x), \cor{v})$. This definition extends to the case of multivariables in the obvious way, and we write $\cor{\varphi(\dot{\vec{x}})}$ for  $\cor{\varphi(\dot{x_1}, \dots, \dot{x_n})}$.  The G\"odelnumbering is canonical, so in particular we require that the following are provable in (a fragment of) \PA:
		\begin{IEEEeqnarray*}{LC}
			\PA\vdash \val(\dot{x}) = x \land \ClTerm(\dot{x}) & \hspace*{50mm}  \\
			\PA\vdash \Fml_1(x) \then\forall z \Sent(x[\dot{z}/v]) \\
			\PA\vdash \ClTerm(x) \land \ClTerm(y) \then (\textsf{Ver}(x \ud= y) \sse \val(x) = \val(y))
		\end{IEEEeqnarray*}
\mypar{Terminology and notation for Gentzen-systems} 
\\
A sequent is a pair $\GD$ of finite sets of \lpat-formulae.
For a finite set of formulae $\Gamma(\vec{x}):=\{\varphi_1(\vec{x}), \dots, \varphi_n(\vec{x})\}$ with free variables contained in $\vec{x}$, $\neg\Gamma(\vec{x})$ denotes the set of negated formulae of $\Gamma$, i.e. $\{\neg\varphi_1(\vec{x}), \dots, \neg\varphi_n(\vec{x})\}$, and $\bigwedge\Gamma(\vec{x})$ denotes the iterated conjunction $\varphi_1(\vec{x}) \land \dots \land \varphi_n(\vec{x})$. If $\Gamma=\emptyset$ we identify $\bigwedge\Gamma$ with $\bo{0 = 0}$. For $t$ free for $v$ in $\Gamma$ (i.e. $t$ is free for $v$ for all members $\varphi_1, \dots, \varphi_n$ of $\Gamma$), we write $\Gamma[t/v]$ for $\{\varphi_1[t/v], \dots, \varphi_n[t/v]\}$. For $\Gamma$ a finite set of \lpat-sentences, we let $T\cor\Gamma:=\{T\cor\gamma\mid\gamma\in\Gamma\}$.
\\
A \textit{derivation} of a sequent \GD is a tree with nodes labelled by sequents. Given a derivation \Dmc and a node \GD of \Dmc, call it $d$, we write $\Gamma'\ert{d}\Delta'$ for $\Gamma', \GD, \Delta'$.
\\
The \textit{height} of a derivation \Dmc is the maximum \textit{lenght} of the branches in the tree, where the lenght of a branch is the number of its nodes minus 1.
\vspace{2mm}
\\
In a rule of inference
	\begin{itemize}
		\item[$\bigcdot$] formulae in $\Gamma, \Delta$ are called \textit{side formulae}, or \textit{context},
		\item[$\bigcdot$] the formulae not in the context in the conclusion are called \textit{principal formulae},
		\item[$\bigcdot$] the formulae in the premises from which the conclusion is derived (i.e. the formulae in the premises not in the context) are called \textit{active formulae}.
	\end{itemize}
A \textit{literal} is an atomic formula or the negation of an atomic formula. The {\it cut rank of a formula} which is eliminated in a cut-rule is the positive complexity of the formula. The supremum of the cut ranks of a derivation \Dmc is called the {\it cut rank} of \Dmc.

\subsection{Sequent calculi for \FDE and some of its extensions}
In this section we introduce the various logics underlying the systems of truth employed in the paper. We start with the two-sided sequent calculus of {\sf F}irst {\sf D}egree {\sf E}ntailment (\FDE). For a general overview of the different non-classical logics employed in this section see \cite{priIntroduction}.

\begin{definition}[\FDE] The logic of \FDE consists of the following axioms and rules.
\vspace*{4mm}
	\begin{IEEEeqnarray*}{CCCC}
		 & \hspace*{40mm} \varphi, \GD, \varphi \hspace{28mm} & 
		\frac{ \GD, \varphi \qquad  \varphi, \GpDp }{ \Gamma', \GD, \Delta' } \hspace{6mm} (\textnormal{Cut}) & \hspace{10mm}
	\end{IEEEeqnarray*}
	\begin{IEEEeqnarray*}{RCCL}
		(\land$\textnormal{L}$) \hspace{10mm} & \frac{ \varphi, \psi, \GD }{ \varphi \land \psi, \GD } \hspace{20mm} &
		\frac{ \GD, \varphi \qquad \GD, \psi }{ \GD, \varphi \land \psi } & (\land$\textnormal{R}$)
		\vspace{3mm}
		\\
		(\forall$\textnormal{L}$) \hspace{10mm} & \frac{ \varphi[t/v], \GD }{ \forall v \varphi, \GD } \hspace{20mm} &
		\frac{ \GD, \varphi[u/v] }{ \GD, \forall v \varphi } & (\forall$\textnormal{R}$)
		\vspace{3mm}
		\\
		(\neg\neg$\textnormal{L}$) \hspace{10mm} & \frac{ \varphi, \GD }{ \neg\neg\varphi, \GD } \hspace{20mm} &
		\frac{ \GD, \varphi }{ \GD, \neg\neg\varphi } & (\neg\neg$\textnormal{R}$)
		\vspace{3mm}
		\\
		(\neg\land$\textnormal{L}$) \hspace{10mm} & \frac{ \neg\varphi, \GD \qquad \neg\psi, \GD }{ \neg (\varphi \land \psi), \GD } \hspace{25mm} &
		\frac{ \GD, \neg\varphi, \neg\psi }{ \GD, \neg(\varphi \land \psi) } &(\neg\land$\textnormal{R}$)
		\vspace{3mm}
		\\
		(\neg\forall$\textnormal{L}$) \hspace{10mm} & \frac{ \neg\varphi[u/v], \GD }{ \neg\forall v \varphi, \GD } \hspace{20mm} &
		\frac{ \Gamma \Then \Delta, \neg\varphi[t/v] }{ \GD, \neg\forall v \varphi } & (\neg\forall$\textnormal{R}$)
	\end{IEEEeqnarray*}
Conditions of application: $\varphi$ literal in initial sequents; $u$ eigenparameter.
\end{definition}

Let
\vspace{2mm}
	\begin{center}
		\Axiom$ \fCenter, \varphi $
		\RightLabel{($\neg$L)}
		\UnaryInf$ \neg\varphi, \fCenter $
		\DisplayProof
		\hspace*{10mm}
		\Axiom$ \varphi, \fCenter $
		\RightLabel{($\neg$R)}
		\UnaryInf$ \fCenter, \neg\varphi $
		\DisplayProof
		\hspace*{10mm}
		\Axiom$ \psi, \fCenter $
		\Axiom$ \neg\psi, \fCenter $
		\RightLabel{(\GG)}
		\BinaryInf$ \varphi,\neg\varphi, \fCenter $
		\DisplayProof
	\end{center}
for $\varphi, \psi$ atomic.

\begin{definition} \hspace{50mm}
	\begin{itemize}
		\item[$\bigcdot$] Classical Logic, \CL, is the system given by \FDE without $\neg\circ$M, for $\circ\in\{\neg, \land, \forall\}, M\in\{L, R\}$, and with the addition of unrestricted \textnormal{($\neg$L)} and \textnormal{($\neg$R)}.
		\item[$\bigcdot$] Strong Kleene, \K, is the system $\FDE +\textnormal{($\neg$L)}$.
		\item[$\bigcdot$] Logic of Paradox, \LP, is the system \FDE $+$ \textnormal{($\neg$R)}.
		\item[$\bigcdot$] Kleene's Symmetric Logic, \KS, is the system $\FDE + (\GG)$.\footnote{
	For similar calculi defining the same logic see, for instance, \cite{scoCombinators} and \cite{blaPartial}.
}
	\end{itemize}
\end{definition}

We now extend the base logics with rules for identity. Let
\vspace{2mm}
	\begin{center}
		\Ax $t = t, \fCenter$
		\RL{Ref}
		\Un $\fCenter$
		\Dis
		\hspace*{15mm}
		\vspace{3mm}
		\\
		\Ax $\varphi(t), \fCenter$
		\RL{RepL}
		\Un $s=t, \varphi(s), \fCenter$
		\Dis
		\hspace{5mm}
		\Ax $\fCenter, \varphi(t)$
		\Ax $\fCenter, s=t$
		\RL{RepR}
		\Bn $\fCenter, \varphi(s)$
		\Dis
	\end{center}

for $\varphi$ literal.

\begin{definition} \hspace{50mm}
	\begin{itemize}
		\item[$\bigcdot$] $\CL_{=}$ is \CL $+$ \textnormal{(Ref)} $+$ \textnormal{(RepL)}
		\item[$\bigcdot$] $\FDE_{=}$ is \FDE $+$ \textnormal{(Ref)} $+$ \textnormal{(RepL)}.
		\item[$\bigcdot$] $\K_{=}$ is \K $+$ \textnormal{(Ref)} $+$ \textnormal{(RepL)}.
		\item[$\bigcdot$] $\LP_{=}$ is \LP $+$ \textnormal{(Ref)} $+$ \textnormal{(RepR)}.
		\item[$\bigcdot$] $\KS_{=}$ is \KS $+$ \textnormal{(Ref)} $+$ \textnormal{(RepL)}.
	\end{itemize}
\end{definition}

\begin{remark}\label{rm_Calculi} \hspace{50mm}
	\begin{itemize}
		\item[$\bigcdot$] (RepL) and (RepR) are equivalent over \FDE, and they both yield the replacement schema
	\[
		s = t, \varphi(s), \GD, \varphi(t)
	\]
		\item[$\bigcdot$] It can easily be shown that \FDE (\KS) and \textnormal{\sf BDM} (\textnormal{\sf SDM}), that is, the system(s) defined by \cite{nicProvably}, are equivalent. The system \KS, however, has the advantage of enjoying a cut elimination theorem.
		\item[$\bigcdot$] The reason for formulating $\K_=$ and $\KS_=$ with \textnormal{(RepL)}, and \LP with \textnormal{(RepR)} is to obtain a syntactic proof of full Cut elimination.\footnote{
	We remark that Cut elimination holds for the sequent calculi introduced above. The key observation is that the calculi defined in this article are designed so that if the cut formula is principal in both premises of Cut, then the complexity of $\varphi$ is >0, i.e. $\varphi$ cannot be a literal. In fact, In $\FDE_{=}, \K_{=}$, and $\KS_{=}$, there is no rule introducing a literal on the right---that is the reason why we formulated $\K_{=}$ and $\KS_{=}$ with (RepL), as they both have one rule introducing literals on the left, i.e. ($\neg$L) and (\GG), respectively---; in $\LP_{=}$ there is no rule introducing a literal on the left---that is the reason why we formulated $\LP_{=}$ via (RepR), as this calculus has one rule introducing literals on the right, i.e. ($\neg$R). Two points are worth emphasizing. The first is that, as far as we know, the rule \GG excluding the simultaneous occurrence of gaps and gluts is new and \KS is the first sequent calculus for symmetric Strong Kleene logic admitting a syntactic proof of Cut elimination. The second is that we are also not aware of Gentzen style calculi for first-order \FDE, \K, \LP, or \KS using so called geometric rules for identity---at least in the literature on truth and paradoxes, the calculi above might be the first axiomatizing identity with rules instead of axioms. A notable exception is \cite{picTruth}, however, Cut is not eliminable in the calculus defined.}
	\end{itemize}
\end{remark}


\section{\KF-like and \PKF-like theories}
\label{sec:KFandPKF}
This section introduces the \KF-like and \PKF-like truth theories. The theory Kripke-Feferman \KF was developed by Feferman \cite[cf.][]{fefReflecting} and further studied by, e.g., \cite{reiRemarks,McGTruth} and \cite{canNotes}. The theory Partial Kripke-Feferman (\PKF) may be seen as the non-classical counterpart to \KF. It was developed by \cite{hahAxiomatizing} and further studied by, e.g., \cite{hanCosts} and \cite{nicProvably}, who introduced the theory $\PKF^+$.\footnote{See \cite{halAxiomatic} for a presentation and discussion of both theories.} 

We first introduce different rules of induction employed in the formulation of the theories. To this end we fix a standard notation system of ordinals up to $\Gamma_0$.\footnote{See
for instance \cite{fefSystems}, \cite{pohProof}.
} 
We use $a, b, c \dots$ to denote the code of our notation system whose value is $\alpha, \beta, \gamma\dots\in On$ (with the exception of $\omega$ and $\varepsilon$-numbers, for which we use the ordinals themselves), and we use $\prec$ to denote the p.r. ordering defined on codes of ordinals. The expression $\forall z\prec y (\varphi[z/v])$ is short for $\forall z (\neg(\textsf{Ord}(z) \land \textsf{Ord}(y) \land z\prec y) \lor \varphi[z/v])$, where \textsf{Ord} represents the set of codes of ordinals. For $\alpha<\Gamma_0$ and a formula $\varphi(v)\in\lpat$ we let $\TI^{<\alpha}(\varphi)$ denote
	\begin{IEEEeqnarray}{c}
		\frac{ \forall z \prec y \ \varphi(z), \GD, \varphi(y) } { \GD, \forall x \prec a \ \varphi(x) } \ztag{$\TI^{<\alpha}$}
	\end{IEEEeqnarray}
We have the following induction schemata ($\varphi(v)\in\lpat$):
	\begin{IEEEeqnarray}{LCL}
		& \hspace{30mm} \frac{ \Gamma, \varphi(u) \Then \varphi(u'), \Delta }{ \Gamma, \varphi(\bo 0) \Then \varphi(t), \Delta } & \hspace{10mm} u\notin FV(\Gamma, \Delta, \varphi(\bo 0)) \ztag{\sf{IND}}
	\end{IEEEeqnarray}
	\begin{IEEEeqnarray}{C}
		\frac{\Fml_1(x), Tx[\dot u / v], \Gamma \Then \Delta, Tx[\dot u' / v]}{Tx[\num(\bo 0) / v], \Gamma \Then \Delta, Tx[\dot z / v]} \ztag{$\textsf{IND}^{\textsf{int}}$}
	\end{IEEEeqnarray}
where $u$ is an eigenvariable and $z$ is an arbitrary term.

We now introduce the basic truth principles employed in the systems of truth we discuss in the paper.

\begin{definition}[Truth Axioms]\label{df_TruthAxioms} The following truth-theoretic initial sequents are called truth axioms. \textnormal{\sf Reg1-2} are called \emph{regularity axioms}.
\allowdisplaybreaks
		\begin{IEEEeqnarray*}{RL}
			(T=) & \hspace{7mm} \ClTerm(x), \ClTerm(y), \val(x) = \val(y), \GD, T(x \ud= y) \\
					& \hspace{7mm} \ClTerm(x), \ClTerm(y), T(x \ud= y), \GD, \val(x) = \val(y) \nonumber \hspace{20mm} \\
%
%
			(T\neg=) & \hspace{7mm} \ClTerm(x), \ClTerm(y), \neg\big(\val(x) = \val(y)\big), \GD, T(\ud\neg (x \ud= y)) \hspace{20mm} \\
					    & \hspace{7mm} \ClTerm(x), \ClTerm(y), T(\ud\neg (x \ud= y)), \GD, \neg\big(\val(x) = \val(y)\big) \nonumber \hspace{20mm} \\
			(\neg T\neg) (i) & \hspace{7mm} \Sent(x), T(\ud\neg x), \GD, \neg T(x) \\
					   (ii) & \hspace{7mm} \Sent(x), \neg T(x), \GD, T(\ud\neg x) \nonumber  \\
			(T\Sent) & \hspace{7mm} Tx, \GD, \Sent(x) 
			\vspace{1mm}
			\\
			(\textnormal{\sf Reg1}) & \hspace{7mm} \Var(z), \ClTerm (y), \Sent(\ud\forall z.x), Tx[y/z], \GD, Tx[\dot{\val(y)}/z] \nonumber\footnotemark\\
			(\textnormal{\sf Reg2}) & \hspace{7mm} \Var(z), \ClTerm (y), \Sent(\ud\forall z.x), Tx[\dot{\val(y)}/z], \GD, Tx[y/z] \nonumber
		\end{IEEEeqnarray*}
\footnotetext{Notice that $\dot{\val(y)} := \num(\val(y))$.}
The following rules are called truth-rules:
	\vspace{2mm}
	\begin{center}
	\small
		\Ax$\ClTerm(x), T(\val(x)), \fCenter$
		\RL{$T$\textnormal{{\sf rp}L}}
		\Un$\ClTerm(x), T(\ud T x), \fCenter$
		\Dis
		\hspace{10mm}
		\Ax$\ClTerm(x), \fCenter, T(\val(x))$
		\RL{$T$\textnormal{{\sf rp}R}}
		\Un$\ClTerm(x), \fCenter, T(\ud T x)$
		\Dis
		\vspace{4mm}
		\\
		\AxC{$\ClTerm(x), T(\ud\neg\val(x)), \GD$}
		\AxC{$\ClTerm(x), \neg\Sent(\val(x)), \GD$}
		\RL{$T$\textnormal{{\sf nrp}L}}
		\BnC{$\ClTerm(x), T(\ud\neg\ud T x), \fCenter$}
		\Dis
		\vspace{2mm}
		\\
		\AxC{$\ClTerm(x), \GD, T(\ud\neg\val(x)), \neg\Sent(\val(x))$}
		\RL{$T$\textnormal{{\sf nrp}R}}
		\UnC{$\ClTerm(x), \GD, T(\ud\neg\ud T x)$}
		\Dis
		\vspace{4mm}
		\\
		\Ax$\Sent(x), T(x), \fCenter$
		\RL{$T$\textnormal{{\sf dn}L}}
		\Un$\Sent(x), T(\ud\neg\ud\neg x), \fCenter$
		\Dis
		\hspace{15mm}
		\Ax$\Sent(x), \fCenter, T(x)$
		\RL{$T$\textnormal{{\sf dn}R}}
		\Un$\Sent(x), \fCenter, T(\ud\neg\ud\neg x)$
		\Dis
		\vspace{4mm}
		\\
		\Ax$\Sent(x\ud\land y), T(x), T(y), \fCenter$
		\RL{$T$\textnormal{{\sf and}L}}
		\Un$\Sent(x\ud\land y), T(x\ud\land y), \fCenter$
		\Dis
		\vspace{2mm}
		\\
		\Ax$\Sent(x\ud\land y), \fCenter, T(x)$
		\Ax$\Sent(x\ud\land y), \fCenter, T(y)$
		\RL{$T$\textnormal{{\sf and}R}}
		\Bn$\Sent(x\ud\land y), \fCenter, T(x\ud\land y)$
		\Dis
		\vspace{4mm}
		\\
		\Ax$\Sent(x\ud\land y), T(\ud\neg x), \fCenter$
		\Ax$\Sent(x\ud\land y), T(\ud\neg y), \fCenter$
		\RL{$T$\textnormal{{\sf nand}L}}
		\Bn$\Sent(x\ud\land y), T(\ud\neg(x\ud\land y)), \fCenter$
		\Dis
		\vspace{2mm}
		\\
		\Ax$\Sent(x\ud\land y), \fCenter, T(\ud\neg x), T(\ud\neg y)$
		\RL{$T$\textnormal{{\sf nand}R}}
		\Un$\Sent(x\ud\land y), \fCenter, T(\ud\neg(x\ud\land y))$
		\Dis
		\vspace{4mm}
		\\
		\Ax$\Sent(\ud\forall v.x), \forall yT(x[\dot{y}/v]), \fCenter$
		\RL{$T$\textnormal{{\sf all}L}}
		\Un$\Sent(\ud\forall v.x), T(\ud\forall v.x), \fCenter$
		\Dis
		\hspace{15mm}
		\Ax$\Sent(\ud\forall v.x), \fCenter, \forall z T(x[\dot{z}/v])$
		\RL{$T$\textnormal{{\sf all}R}}
		\Un$\Sent(\ud\forall v.x), \fCenter, T(\ud\forall v.x)$
		\Dis
		\vspace{4mm}
		\\
		\Ax$\Sent(\ud\forall v.x), \neg\forall z\neg T(\ud\neg x[\dot{z}/v]), \fCenter$
		\RL{$T$\textnormal{{\sf nall}L}}
		\Un$\Sent(\ud\forall v.x), T(\ud\neg\ud\forall v.x), \fCenter$
		\Dis
		\hspace{15mm}
		\Ax$\Sent(\ud\forall v.x), \fCenter, \neg\forall z\neg T(\ud\neg x[\dot{z}/v])$
		\RL{$T$\textnormal{{\sf nall}R}}
		\Un$\Sent(\ud\forall v.x), \fCenter, T(\ud\neg\ud\forall v.x)$
		\Dis
	\end{center}
\end{definition}
\vspace{4mm}

The following definitions introduce the various \KF- and \PKF-style theories.

\begin{definition}[\KF] \KF is obtained from $\CL_=$ by adding sequents $\Then \varphi$ for $\varphi$ axiom of \PA, (\textnormal{\sf IND}), truth axioms and truth-rules of Df~\ref{df_TruthAxioms}, except the axiom ($\neg T\neg$).
\end{definition}

\begin{definition}[\KF-variants]\label{df_KFCluster} We introduce variants of \KF:
		\begin{itemize}
			\item[(i)] $\KFcs$ is obtained from \KF by adding \Cons, i.e., $\Sent(x), T(\ud\neg x), \GD, \neg T(x)$.
			\item[(ii)] $\KFcp$ is obtained from \KF by adding \Comp, i.e., $\Sent(x), \neg T(x), \GD, T(\ud\neg x)$.\footnote{%
				Note that \Cons is $(\neg T\neg)(i)$ and \Comp is $(\neg T\neg)(ii)$. Of course, over the nonclassical logics studied in this paper, $(\neg T\neg)$ does not imply that the truth predicate is consistent and complete. $(\neg T\neg)$ is just axiomatizing the well known property of fixed-point models that the anti-extension $A$ can be defined via the extension $E$ as $A:=\{\varphi\mid\neg\varphi\in E\}$.
			}
			\item[(iii)] $\KFs$ is obtained from \KF by adding
				\begin{IEEEeqnarray}{C}
					\Sent(x), \Sent(y), T(x), T(\ud\neg x), \GD, T(y), T(\ud\neg y) \ztag{\textnormal{{\sf GoG}}}
				\end{IEEEeqnarray}
		\end{itemize}
		Let $\KF_\star\in\{\KF, \KFcs, \KFcp, \KFs\}$. Then
		\begin{itemize}
			\item[(iv)] $\KF_\star\har$ is obtained from $\KF_\star$ by restricting \textnormal{({\sf IND})} on $\lpa$-formulae.
			\item[(v)] $\KF_\star^{\textnormal{\sf int}}$ is the theory obtained from $\KF_\star\har$ by replacing the restricted version of \textnormal{({\sf IND})} with $\textnormal{({\sf IND}}^\textnormal{\sf int})$.
		\end{itemize}
\end{definition}

\begin{fact}\label{ob_Embedding}
Let $\KF_\star\in\{\KF, \KFcs, \KFcp, \KFs\}$. Then \KFst\har is contained in the subtheory \KFintst\har of \KFintst, obtained by allowing (\textnormal{\sf IND}$^\textnormal{\sf int}$) only in the following restricted version:
	 \begin{prooftree}
		\Ax$\Fml_{\lpa}(x), Tx[\dot{u}/v], \fCenter, Tx[\dot{u'}/v]$
		\RightLabel{\textnormal{\footnotesize{(\textnormal{\sf IND}$^\textnormal{\sf int}$\hspace{1mm}\har\lpa)}}}
		\Un$Tx[\bo 0/v], \fCenter, Tx[\dot{z}/v]$
	\end{prooftree}
for $u$ eigenvariable and $z$ arbitrary.
\end{fact}

\begin{remark}\label{rm_Literals}
	Let $\KFintst\in\{\KFint, \KFintcs, \KFintcp, \KFints\}$. Every axiom of \KFintst has the form $\Theta, \GD, \Lambda$. Formulae in $\Theta, \Lambda$ are called \emph{active}. Every active formula has positive complexity $0$.\footnote{%
	Literals have positive complexity $=0$.
}
\end{remark}

Call a derivation \Dmc quasi-normal if \Dmc has cut rank 0. By application of standard techniques for Cut-elimination for predicate calculus, we get

\begin{proposition}\label{pr_CutEliminationKFint}
	Let \KFintst be as in Remark~\ref{rm_Literals}. Then every \KFintst-derivation \Dmc can be transformed into a \emph{quasi-normal} derivation $\Dmc'$ with the same end sequent.
\end{proposition}

\begin{proposition}[Inversion]\label{pr_HPInversion} Let \KFintst be as in Remark~\ref{rm_Literals}. Then
	\begin{enumerate}[label=(\roman*)]
		\item If $\KFintst\vdash^n\GD, \forall v\varphi$, then $\KFintst\vdash^n\GD, \varphi[u/v]$ for any $u\notin FV(\Gamma, \Delta, \varphi)$.
		\item If $\KFintst\vdash^n\forall v\varphi, \GD$, then $\KFintst\vdash^n\varphi[t/v], \GD$ for some term $t$.
	\end{enumerate}
\end{proposition}

We move on to \PKF-like theories.
\begin{definition}[\PKF]\label{df_PKF} \PKF is obtained from $\FDE_=$ by adding sequents $\GD, \varphi$ and $\neg\varphi,\GD$ for $\varphi$ axiom of \PA, (\textnormal{\sf IND}), truth axioms and truth-rules of Df.~\ref{df_TruthAxioms}, and the following two rules requiring identity statements to behave classically:\footnote{
	Adding $\neg\varphi, \GD$ for $\varphi$ axiom of \PA to $\PKF$-like theories makes contraposition admissible in \PKF and \PKFs (see Lemma~\ref{le_Contraposition}).
}
\vspace{2mm}
	\begin{center}
		\AxC{$\GD, s=t$}
		\RL{\textnormal{$=\neg$L}}
		\UnC{$\neg(s=t), \GD$}
		\Dis
		\hspace{15mm}
		\AxC{$ s=t, \GD$}
		\RL{\textnormal{$=\neg$R}}
		\UnC{$\GD, \neg(s=t)$}
		\Dis
	\end{center}

\end{definition}

\begin{definition}[\PKF-cluster]\label{df_PKFCluster} 
	We introduce variants of \PKF
		\begin{itemize}
			\item[(i)] $\PKFcs$ is obtained by adding \textnormal{($\neg$L)} to \PKF.
			\item[(ii)] $\PKFcp$ is obtained by adding \textnormal{($\neg$R)} to \PKF.
			\item[(iii)] $\PKFs$ is obtained by adding \textnormal{(\GG)} to \PKF.
		\end{itemize}
		Let $\PKF_\star\in\{\PKF, \PKFcs, \PKFcp, \PKFs\}$. Then
		\begin{itemize}
			\item[(iv)] $\PKF_\star\har$ is obtained from $\PKF_\star$ by restricting \textnormal{({\sf IND})} on $\lpa$-formulae.
			\item[(v)] $\PKF_\star^+$ is obtained by extending $\PKF_\star$ with $(\TI^{<\varepsilon_0})$.
		\end{itemize}
\end{definition}

Since our formulation of \PKF deviates from the formulation in \cite{hahAxiomatizing}, we now show that \PKF behaves classically on the $T$-free fragment of \lpat and that $\psi(\vec{x})$ and $T\cor{\psi(\dot{\vec{x}})}$ are interderivable.

\begin{lemma}\label{le_PKF}
	Let $\PKF_\star$ be one of the \PKF-like theories introduced in Df.~\ref{df_PKFCluster}, $\varphi\in\lpa$, and $\psi(\vec{x})\in\lpat$. Then
		\begin{enumerate}[label=(\roman*)]
			\item $\PKF_\star\vdash\GD, \varphi, \neg\varphi$ and $\PKF_\star\vdash\varphi, \neg\varphi, \GD$;
			\item $\PKF_\star\vdash\GD, \psi(\vec{x})$ iff $\PKF_\star\vdash\GD, T\cor{\psi(\dot{\vec{x}})}$;
			\item Unrestricted \emph{($\neg$L)} and \emph{($\neg$R)} are admissible for $\varphi\in\lpa$.
		\end{enumerate}
\end{lemma}

\begin{proof}
	(i) and (ii) are shown by a straightforward induction on $\varphi$. For (iii), observe more generally that, if $\varphi, \neg\varphi, \GD$ and $\GD, \varphi, \neg\varphi$ are both derivable, then ($\neg$L) and ($\neg$R) are derived rules
\vspace{2mm}
	\begin{center}
		\Ax $\fCenter, \varphi$
		\Ax $\varphi, \neg\varphi, \fCenter$		
		\RL{Cut}
		\Bn $\neg\varphi, \fCenter$
		\Dis
		\hspace{15mm}
		\Ax $\fCenter, \neg\varphi, \varphi$
		\Ax $\varphi, \fCenter$		
		\RL{Cut}
		\Bn $\fCenter, \neg\varphi$
		\Dis
	\end{center}
\end{proof}

Finally, to complete the picture we note that contraposition is admissible in \PKF  and \PKFs.

\begin{lemma}[Contraposition]\label{le_Contraposition}
	Contraposition, i.e.~the rule
	\[
		\frac{ \GD }{ \neg\Delta\Then\neg\Gamma }
	\]
is admissible in \PKF ($\PKF\har,\PKF^+$) and \PKFs ($\PKFs\har,\PKFs^+$).
\end{lemma}

\begin{proof} 
The proof is by induction on the height of derivations. We show two crucial cases involving the rules (\GG) and (RepL). Suppose the derivation ends with	
	\begin{prooftree}
		\AxC{}
		\RightLabel{$\Dmc_0$}
		\branchDeduce
		\Deduce$\psi, \fCenter$
		\AxC{}
		\RightLabel{$\Dmc_1$}
		\branchDeduce
		\Deduce$\neg\psi, \fCenter$
		\RL{\GG}
		\Bn$\varphi, \neg\varphi, \fCenter$
	\end{prooftree}	
We have
	\begin{prooftree}
	\footnotesize
	\def\fCenter{\mbox{$\neg\Delta \Then \neg\Gamma$}}
		\AxC{}
		\RL{i.h.}
		\branchDeduce
		\Deduce$\fCenter, \neg\psi$
		\AxC{}
		\RL{i.h.}
		\branchDeduce
		\Deduce$\fCenter, \neg\neg\psi$
		\RL{Inv\footnotemark}
		\Un$\fCenter, \psi$
		\AxC{$\varphi, \neg\Delta \Then \neg\Gamma, \varphi, \neg\varphi$ }
		\AxC{$\neg\varphi, \neg\Delta \Then \neg\Gamma, \varphi, \neg\varphi$ }
		\RL{\GG}
		\BnC{$\psi, \neg\psi, \neg\Delta \Then \neg\Gamma, \varphi, \neg\varphi$}
		\RL{Cut}
		\BnC{$\neg\psi, \neg\Delta \Then \neg\Gamma, \varphi, \neg\varphi$}
		\RL{Cut}
		\Bn$\fCenter, \varphi, \neg\varphi$
		\RL{$\neg\neg$R}
		\Un$\fCenter, \neg\varphi, \neg\neg\varphi$
	\end{prooftree}	
\footnotetext{
	Notice that in \PKF we have invertibility of ($\neg\neg$R).
}
If the derivations ends with
	\begin{prooftree}
		\AxC{}
		\RightLabel{$\Dmc_0$}
		\branchDeduce
		\Deduce$\varphi(t), \fCenter$
		\RL{RepL}
		\Un$s=t, \varphi(s), \fCenter$ 
	\end{prooftree}	
we reason as follows. We first derive
	\begin{prooftree}
	\def\fCenter{\mbox{$\neg\Delta \Then \neg\Gamma$}}
		\AxC{}
		\RL{i.h.}
		\branchDeduce
		\Deduce$\fCenter, \neg\varphi(t)$
		\Ax$\neg\varphi(s), \fCenter, \neg\varphi(s)$
		\RL{RepL}
		\Un$s=t, \neg\varphi(t), \fCenter, \neg\varphi(s)$
		\RL{Cut}
		\Bn$s=t, \fCenter, \neg\varphi(s)$
		\RL{$=\neg$R}
		\Un$\fCenter, \neg\varphi(s), \neg(s=t)$
	\end{prooftree}
\end{proof}


\section{Reinhardt's Challenge}
\label{sec:InternalLogic}
In this section we address {\sc Reinhardt's Challenge} and show that given a \KF-like theory there is a corresponding \PKF-like theory such that the set of inferences $\TGD$ provable in the \PKF-like theory coincides with the set of significant inferences of the \KF-like theory. This observation may be considered as a positive answer to Question b) discussed in the Introduction, that is, as providing an independent axiomatization of the significant inferences of the \KF-like theories. Moreover, by axiomatizing the set of significant inferences of \KF-like theories we obtain a positive answer to the {\sc Generalized Reinhardt Problem} as an immediate corollary. More precisely, as shown in Proposition \ref{prop_ReinKF},  every significant inference of a \KF-like theory has a significant derivation, i.e., whenever  the theory proves $\TGD$, we can find a derivation \Dmc of \TGD such that every node of \Dmc is itself a significant inference.

To begin with define the notion of significant inferences for arbitrary truth theories.

\begin{definition}[Significant Inferences]\label{df_SignificantInferences}
	Let $\mathrm{Th}$ be an axiomatic truth-theory formulated in $\lpat$, and $\Gamma,\Delta$ be finite sets of \lpat-sentences. The set of \textit{significant inferences} of $\mathrm{Th}$ is defined as
		\begin{IEEEeqnarray*}{C}
			\mathrm{Th}\Ssf\Isf:=\{ \ang{\Gamma,\Delta} \mid \mathrm{Th}\vdash \TGD \}.
		\end{IEEEeqnarray*}
\end{definition}


Since the truth predicate of \PKF-like theories is transparent (cf.~Lemma~\ref{le_PKF}(ii)) the significant inferences of any \PKF-like theory will simply amount to the set of provable inferences of the theory. We also note in passing that the significant part of a truth theory ($\mathrm{Th}\Ssf$) in the sense of \cite{reiRemarks85,reiRemarks}, that is the provably true sentences of the theory, can be retrieved from $\mathrm{Th}\Ssf\Isf$ by setting\begin{IEEEeqnarray*}{C}
			\mathrm{Th}\Ssf:=\{ \varphi\in\Sent \mid \ang{\emptyset,\varphi}\in\mathrm{Th}\mathsf{SI} \}.
		\end{IEEEeqnarray*}

Let us now show that for each \KF-like theory there is a \PKF-like theory such that the provable sequents of the latter constitute exactly the significant inferences of the former.

\begin{definition}[$\PKF^\circ, \KF^\circ$]\label{pairs}The pair $(\PKF^\circ, \KF^\circ)$ is a variable ranging over the following theory-pairs
	\begin{IEEEeqnarray*}{C}
		(\PKF\har, \KF\har), (\PKFcs\har, \KFcs\har), (\PKFcp\har, \KFcp\har), (\PKFs\har, \KFs\har) \\
		(\PKF, \KFint), (\PKFcs, \KFintcs), (\PKFcp, \KFintcp), (\PKFs, \KFints)
		\\
		(\PKF^+, \KF), (\PKFcs^+, \KFcs), (\PKFcp^+, \KFcp), (\PKFs^+, \KFs).\end{IEEEeqnarray*}
Moreover, let $\mathrm{Th}\in\{\KF,\KFint,\KF\har,\PKF^+,\PKF\har\}$. Then $\mathrm{Th}_\star\in\{\mathrm{Th},\mathrm{Th}_\mathsf{S}, \mathrm{Th}_\mathsf{cs},\mathrm{Th}_\mathsf{cp}\}.$ 
\end{definition}

We can now start proving the principal result of this section, i.e., we can prove that $\PKF^\circ=\KF\Ssf\Isf^\circ$. We first show that $\PKF^\circ\Sub\KF\Ssf\Isf^\circ$.

\begin{proposition}\label{le_H&HNExt}Let $(\PKF^\circ, \KF^\circ)$ as defined in Definition \ref{pairs}. Then
		\begin{IEEEeqnarray*}{C}
			\text{If } \PKF^\circ\vdash\GDv, \ \text{ then } \
			\KF^\circ\vdash \TGDv.
		\end{IEEEeqnarray*}\end{proposition}
Proposition \ref{le_H&HNExt} is essentially due to \cite{hahAxiomatizing}, \cite{hanCosts} and \cite{nicProvably}, who proved the claim for theories without index {\sf cs} or {\sf cp}, that is, for pairs of theories that do not assume the truth predicate to be consistent or complete. It thus suffices to extend their result to these theories.

\begin{proof}[Proof of Proposition \ref{le_H&HNExt}]
	The proof is straightforward.\footnote{Notice that the use of rules for identity instead of identity axioms does not impact the arguments due to \cite{hahAxiomatizing}, \cite{hanCosts} and \cite{nicProvably}.} For pairs extended with a consistency principle, it suffices to show that the \KF-theory ``internalizes'' the soundness of ($\neg$L). That is, it suffices to show that, e.g., if $\KFcs\vdash T\cor{\Gamma} \Then T\cor{\Delta}, T\cor\varphi$, then $\KFcs\vdash T\cor{\neg\varphi}, T\cor\Gamma \Then T\cor{\Delta}$. Symmetrically for theories extended with a completeness principle one needs to show that the \KF-theory ``internalizes'' the soundness of ($\neg$R).
\end{proof}
In light of the definition of $\KF^\circ\Ssf\Isf$, Proposition \ref{le_H&HNExt} immediately yields that the provable inferences of $\PKF^\circ$ are a subset $\KF^\circ\Ssf\Isf$:

\begin{corollary}\label{co_ILPKFSubILKF}
	$\PKF^\circ\Sub\KF^\circ\Ssf\Isf$.
\end{corollary}


\subsection{From $\KF^\circ$-significant inferences to $\PKF^\circ$-provable sequents}\label{subsec:KFinPKF}
The proof of the converse directions of Proposition \ref{le_H&HNExt} and Corollary \ref{co_ILPKFSubILKF} constitutes the main technical contribution of this article. The basic idea underlying the proof is to show by induction on the height of $\KF^\circ$-derivations that if $\KF^\circ\vdash\TGD$ then $\PKF^\circ\vdash\TGD$. This will be shown by proving a stronger claim, i.e. it will be shown that, whenever $\Gamma, \Delta$ contain only literals
	\[
		\KF^\circ\vdash\GD \ \text{ implies } \ \PKF^\circ\vdash \Gamma^+,\Delta^-\Rightarrow\Gamma^-,\Delta^+
	\]
where for $\Theta\in\{\Gamma, \Delta\}$ and {\sf At} the set of atomic sentences,
	\begin{IEEEeqnarray*}{lr}
		\Theta^+:=\{\varphi\in\textsf{At}\mid\varphi\in\Theta\} & \hspace{6mm} \Theta^-:=\{\varphi\in\textsf{At}\mid\neg\varphi\in\Theta\}\footnote{%
			In other words, $\psi\in\Theta^-$ iff $\neg\psi\in\Theta$.
		}
	\end{IEEEeqnarray*} 
This transformation is motivated by (i) the fact that identity behaves classically in $\PKF^\circ$ and (ii) the following semantic consideration: if a formula of the form $\neg Tt$ is classically false (true), then $Tt$ is either true (false) or both (neither) from the perspective of the non-classical theory of truth. As a consequence, moving $Tt$ in the succedent (antecedent) of the sequent will not interfere with the validity of the sequent from the perspective of the non-classical logics at stake. Moreover, if  $\GD$ is of the form $T\cor{\Gamma'}\Rightarrow T\cor{\Delta'}$, the transformation leaves the sequent unaltered, hence, if we prove that for each $\KF^\circ$-provable sequent its transformation is $\PKF^\circ$-provable, we obtain our desired result as a corollary.

\begin{lemma}[Main Lemma]\label{le_TKFintSubPKF}
	For $\Gamma, \Delta$ sets of literals, for all $\star$
	\[
		\KFintst\vdash\GD \ \text{ implies } \PKFst\vdash\Gamma^+, \Delta^-\Then\Gamma^-, \Delta^+.
	\]
\end{lemma}
\begin{remark} In the following proof we implicitly use the facts that
	\begin{itemize}
		\item[$\bigcdot$] the replacement schema $s=t, \varphi(t), \GD, \varphi(s)$ is derivable in \PKF-like systems;
		\item[$\bigcdot$] if $\KFintst\vdash^n\GD,\neg\forall v\varphi$, then $\KFintst\vdash^n\forall v\varphi, \GD$, and if $\KFintst\vdash^n\neg\forall v\varphi, \GD$, then $\KFintst\vdash^n\GD, \forall v\varphi$.\footnote{%
	Note that $\neg\forall v\varphi$ can be principal on the right (left) only if it has been derived via $\neg$R ($\neg$L).
}
	\end{itemize}
\end{remark}
\begin{proof}[Proof of Lemma~\ref{le_TKFintSubPKF}]
	The proof is by induction on the height $n$ of \KFintst-derivations.
\vspace{1mm}
\\
$\boxed{n=0}$ \hspace{1mm} Suppose \GD is a \KFintst-initial sequent. For the pair \KFint-\PKF, we first notice that every axiom of \KFint is an axiom of \PKF. Hence for initial sequents not involving negated atomic formulae, the proof is immediate. The only truth-theoretic axiom of \KFint involving a negated atomic formula is ($T\neg=$). We obtain the desired conclusion by application of Lemma~\ref{le_PKF}(iii), e.g.
	\begin{prooftree}
		\AxC{$\ClTerm(x), \ClTerm(y), \neg(\val(x)=\val(y)), \Gamma^+, \Delta^-\Then\Gamma^-, \Delta^+, T(\ud\neg(x\ud = y))$}
		\RL{Lemma~\ref{le_PKF}(iii)}
		\UnC{$\ClTerm(x), \ClTerm(y), \Gamma^+, \Delta^-\Then\Gamma^-, \Delta^+, T(\ud\neg(x\ud = y)), \val(x)=\val(y)$}
	\end{prooftree}
$\boxed{\Cons, \Comp, \textsf{GoG}}$ \hspace{1mm} As for theories with specific $T$-axioms, we want to show (we omit context for readability)
	\begin{IEEEeqnarray}{C}
		\PKFcp\vdash \Sent(x)\Then T(x), T(\ud\neg x) \label{ty'} \\
		\PKFcs\vdash \Sent(x), T(x), T(\ud\neg x) \Then \label{tyy'} \\
		\PKFs\vdash \Sent(x), \Sent(y), T(x), T(\ud\neg x) \Then T(y), T(\ud\neg y) \label{tyyy'}
	\end{IEEEeqnarray}
For (\ref{ty'}) 
	\begin{prooftree}
	\def\fCenter{\mbox{ $\Then$ }}
		\Ax$\Sent(x), \neg T(x) \fCenter T(\ud\neg x)$
		\RL{$\neg$R}
		\Un$\Sent(x) \fCenter T(\ud\neg x), \neg\neg T(x)$
		\Ax$\Sent(x), T(x) \fCenter T(x)$
		\RL{$\neg\neg$L}
		\Un$\Sent(x), \neg\neg T(x) \fCenter T(x)$
		\RL{Cut}
		\Bn$\Sent(x) \fCenter T(x), T(\ud\neg x)$
	\end{prooftree}
For (\ref{tyy'}) we have
	\begin{prooftree}
	\def\fCenter{\mbox{ $\Then$ }}
		\Ax$\Sent(x), T(x) \fCenter T(x)$
		\RL{$\neg$L}
		\Un$\Sent(x), T(x), \neg T(x) \fCenter$
		\Ax$\Sent(x), T(\ud\neg x) \fCenter \neg T(x)$
		\RL{Cut}
		\Bn$\Sent(x), T(x), T(\ud\neg x) \fCenter$
	\end{prooftree}
Finally, for (\ref{tyyy'})
	\begin{prooftree}
	\def\fCenter{\mbox{ $\Then$ }}
		\Ax$\Sent(x), \Sent(y), T(y) \fCenter, T(y) \neg T(y)$
		\Ax$\Sent(x), \Sent(y), \neg T(y) \fCenter T(y), \neg T(y)$
		\RL{\GG}
		\Bn$\Sent(x), \Sent(y), T(x), \neg T(x) \fCenter T(y), \neg T(y)$
		\RL{$\neg T\neg$}
		\Un$\Sent(x), \Sent(y), T(x), T(\ud\neg x) \fCenter T(y), T(\ud\neg y)$
	\end{prooftree}
$\boxed{n=m+1}$ \hspace{1mm} Suppose \GD has been derived. We distinguish two cases: either \GD contains a principal formula, or it contains no principal formula. If the latter, then it has been derived either by (Ref), in which case we just apply i.h. and (Ref) in \PKFst, or it has been derived by 
\vspace{1mm}
\\
$\boxed{\text{Cut}}$ \hspace{1mm} For this to work, it is crucial that we are dealing with quasi-normal derivations. In this case (Cut) is applied to literals only and we can just apply (Cut) in \PKFst. If $\varphi\in\textsf{At}$, we have
	\begin{prooftree}
	\def\fCenter{\mbox{ $\Gamma^+, \Delta^-\Then\Gamma^-, \Delta^+$}}
		\AxC{}
		\RL{i.h.}
		\branchDeduce
		\Deduce$\fCenter, \varphi$
		\AxC{}
		\RL{i.h.}
		\branchDeduce
		\Deduce$\varphi, \fCenter$
		\RL{Cut}
		\Bn$\fCenter$
	\end{prooftree}
If $\varphi\equiv\neg\psi\in\textsf{At}^-$, where $\textsf{At}^-$ is the set of negated atomic formulae, then we use i.h. and cut on $\psi$.
\vspace{1mm}
\\
Now suppose $\GD$ contains a principal formula. The rules having literals as principal formulae are the following: ($\neg$L), ($\neg$R), (RepL), ({\sf IND}$^\textsf{int}$), and truth-rules.
\vspace{1mm}
\\
$\boxed{\neg$L$, \neg$R$}$ \hspace{1mm} Immediate, e.g. suppose the \KFintst-derivation ends with
	\begin{prooftree}
	\def\fCenter{\mbox{ $\GpD$}}
		\AxC{}
		\branchDeduce
		\Deduce$\fCenter, \varphi$
		\RL{$\neg$L}
		\Un$\neg\varphi, \fCenter$
	\end{prooftree}
If $\varphi\in\textsf{At}$, then by induction we have $\PKFst\vdash\Gamma'^+, \Delta^-\Then\Gamma'^-, \Delta^+, \varphi$, which is our desired conclusion. The case where $\varphi\equiv\neg\psi\in\negat$ need not be taken into account, as we are dealing with derivations containing only literals in the end-sequent.
\vspace{1mm}
\\
$\boxed{\text{RepL}}$ \hspace{1mm} Suppose the \KFintst-derivation ends with
	\begin{prooftree}
	\def\fCenter{\mbox{ $\GpD$}}
		\AxC{}
		\branchDeduce
		\Deduce$\varphi(t), \fCenter$
		\Un$s=t, \varphi(s), \fCenter$
	\end{prooftree}
Suppose first that $\varphi\in\textsf{At}$. For the pairs \KFint-\PKF, \KFintcs-\PKFcs, and \KFints-\PKFs, we just apply i.h. and (RepL) in the \PKF-variant. For the pair \KFintcp-\PKFcp, we reason in \PKFcp as follows
	\begin{prooftree}
		\AxC{}
		\RL{Rep Schema}
		\UnC{$s = t, \varphi(s), \Gamma'^+, \Delta^-\Then\Gamma'^-, \Delta^+, \varphi(t)$}
		\AxC{}
		\RL{i.h.}
		\shortDeduce
		\DeduceC{$\varphi(t), \Gamma'^+, \Delta^-\Then\Gamma'^-, \Delta^+$}
		\RL{Cut}
		\BnC{$s=t, \varphi(s), \Gamma'^+, \Delta^-\Then\Gamma'^-, \Delta^+$}
	\end{prooftree}
Now suppose $\varphi\equiv\neg\psi\in\textsf{At}^-$. We first reason in an arbitrary \PKFst
	\begin{prooftree}
		\AxC{}
		\RL{i.h.}
		\shortDeduce
		\DeduceC{$\Gamma'^+, \Delta^-\Then\Gamma'^-, \Delta^+, \psi(t)$}
		\AxC{}
		\RL{Rep Schema}
		\UnC{$s=t, \psi(t), \Gamma'^+, \Delta^-\Then\Gamma'^-, \Delta^+, \psi(s)$}
		\RL{Cut}
		\BnC{$s=t, \Gamma'^+, \Delta^-\Then\Gamma'^-, \Delta^+, \psi(s)$}
	\end{prooftree}
\vspace{1mm}
$\boxed{\textsf{IND}^\textsf{int}}$ \hspace{1mm} Suppose the \KFintst-derivation ends with
	\begin{prooftree}
		\Ax$\Fml_1(x), T(x[\dot{u}/v]), \fCenter, T(x[\dot{u'}/v])$
		\RL{{\sf IND}$^\textsf{int}$}
		\Un$T(x[\bo 0/v]), \fCenter, T(x[\dot{z}/v])$
	\end{prooftree}
with $u$ eigenvariable and $z$ arbtrary term. We use ({\sf IND}) in \PKFst as follows
	\begin{prooftree}
	\def\fCenter{\mbox{ $\Gamma^+, \Delta^-\Then\Gamma^-, \Delta^+$}}
		\AxC{}
		\RL{i.h}
		\shortDeduce
		\Deduce$\Fml_1(x), T(x[\dot{u}/v]), \fCenter, T(x[\dot{u'}/v])$
		\RL{{\sf IND}}
		\Un$T(x[\bo 0/v]), \fCenter, T(x[\dot{z}/v])$
	\end{prooftree}
Logical rules for $\land$ and $\forall$ need not be taken into account as they cannot be the last rule of a quasi-normal derivation \Dmc that has in the end-sequent. We end the proof by dealing with truth-rules. We first note that, with the exception of truth-rules for the universal quantifier, all other truth-rules have only literals as active formulae. Hence, it suffices to apply i.h. and the rule itself in \PKFst. As for rules involving $\forall$, we exploit height-preserving inversion (see Proposition~\ref{pr_HPInversion}). For example
\vspace{1mm}
\\
$\boxed{T\text{\sf all}\text{R}}$ \hspace{1mm} Suppose the derivation end with 
	\begin{prooftree}
	\def\fCenter{\mbox{ $\GDp$}}
		\Ax$\fCenter, \forall z T(x[\dot{z}/v])$
		\RL{T{\sf nall}R}
		\Un$\fCenter, T(\ud\forall v.x)$
	\end{prooftree}
By inversion, from $\KFintst\vdash^n\GD, \forall z T(x[\dot{z}/v])$, we get $\KFintst\vdash^n\GDp, Tx[\dot{u}/v]$ for $u$ eigenvariable. Using i.h., we then reason in \PKFst as follows
	\begin{prooftree}
	\def\fCenter{\mbox{ $\Gamma^+, \Delta'^-\Then\Gamma^-, \Delta'^+$}}
		\AxC{}
		\RL{i.h.}
		\shortDeduce
		\Deduce$\fCenter, Tx[\dot{u}/v]$
		\RL{$\forall$R}
		\Un$\fCenter, \forall z Tx[\dot{u}/v]$
		\RL{$T${\sf all}R}
		\Un$\fCenter T(\ud\forall v.x)$
	\end{prooftree}
A similar argument works for $T\textsf{all}L$.
\vspace{1mm}
\\
$\boxed{T\text{\sf nall}\text{R}}$ \hspace{1mm} Suppose the derivation ends with 
	\begin{prooftree}
	\def\fCenter{\mbox{ $\GDp$}}
		\Ax$\fCenter, \neg\forall z\neg T(\ud\neg x[z/v])$
		\RL{$T${\sf nall}R}
		\Un$\fCenter, T(\ud\neg\ud\forall v.x)$
	\end{prooftree}
From $\KFintst\vdash^n\GDp,\neg\forall z\neg T(\ud\neg x[\dot{z}/v])$, we get $\KFintst\vdash^n\forall z\neg T(\ud\neg x[\dot{z}/v]), \GDp$ and by inversion we get $\KFintst\vdash^n\neg T(\ud\neg x[\dot{y}/v]), \GDp$, for some $y$. Using i.h., we then reason in \PKFst as follows
	\begin{prooftree}
	\def\fCenter{\mbox{ $\Gamma^+, \Delta'^-\Then\Gamma^-,\Delta'^+$}}
		\Ax$\fCenter, T(\ud\neg x[\dot{y}/v])$
		\RL{$\neg\neg$R}
		\Un$\fCenter, \neg\neg T(\ud\neg x[\dot{y}/v])$
		\RL{$\neg\forall$R}
		\Un$\fCenter, \neg\forall z \neg T(\ud\neg x[\dot{z}/v])$
		\RL{$T${\sf nall}R}
		\Un$\fCenter, T(\ud\neg\ud\forall v.x)$
	\end{prooftree}
\end{proof}

By inspecting the proof of the Main Lemma it is immediate that we can lift the lemma to the pair  $\KFintst\har$ and $\PKFst\har$, that is, we obtain the following

\begin{corollary}\label{TKFrestrPKFrestr} For $\Gamma, \Delta$ sets of literals, for all $\star$
		\[
		 	\KFintst\har\vdash\GD \ \text{ implies } \ \PKFst\har\vdash\Gamma^+,\Delta^-\Then\Gamma^-, \Delta^+.
		 \]
\end{corollary}

We also note that the Lemma~\ref{le_TKFintSubPKF} can be formalized in \PKFst, which will prove important when we consider \KF-systems with full induction. The relation $\KFintst\vdash^n_\rho\GD$ expressing that \GD is derivable in \KFintst with a derivation of height $\leq n$ and cut rank $\leq\rho$ can be canonically represented in arithmetic via a recursively enumerable predicate $\Bew_{\KFintst}(\bo n, \bo r, \cor{\GD})$.

\begin{corollary}\label{co_TKFintSubPKF} For $\Gamma, \Delta$ sets of literals, for all $\star$
\begin{align*} \PKF_\star\vdash\Bew_{\KFintst}(\bo n, \bo 0, \cor{\GD})&\text{ implies }\PKF_\star\vdash\Gamma^+, \Delta^-\Then\Gamma^-, \Delta^+.\end{align*}
\end{corollary}

We have shown how to transform provable sequents in \KF-systems with restricted or internal induction into provable sequents of appropriate \PKF-like theories. It remains to show the claim for \KF-like theories with full induction. For these theories, cut elimination fails. However, it is well known that axioms and rules \KF-like theories can be derived in a sequent calulus with infinitary rules replacing the schema of induction. Hence we can take a detour via the infinitary system $\KF^\infty$, which contains \KF and enjoys partial cut elimination. The key observation that makes this detour possible is that for $\KF^\infty$ quasi-normal derivations up to height $\varepsilon_0$ the strategy employed in our Main lemma (Lemma \ref{le_TKFintSubPKF}) can be used 
to provide suitable $\PKF^+$-derivations. This proves sufficient for lifting Lemma \ref{le_TKFintSubPKF} to the $\KF$-theories with full induction. The technique for embedding of $\KF$ in $\KF^\infty$ is well known, however, we give a succinct presentation of $\KF^\infty$ for the sake of completeness.\footnote{%
	See, for instance, \cite{takProof} or \cite{sepProof}.
}

\begin{definition}[$\PA_\omega$]
	The language of the infinitary sequent-style version of \PA, $\PA_\omega$, is obtained by omitting free variables. The axioms of $\PA_\omega$ are
	\begin{IEEEeqnarray*}{RL}
		\emptyset\Then\varphi & \hspace{2mm} \text{ if } \varphi \text{ is a true atomic sentence;} \\
		\varphi\Then\emptyset & \hspace{2mm} \text{ if } \varphi \text{ is a false atomic sentence;} \\
		\varphi(t)\Then\varphi(s) & \hspace{2mm} \text{ if } \varphi \text{ is an atomic sentence and } t_i \text{ and } s_i \text{ evaluate to the same numeral.}
	\end{IEEEeqnarray*}
$\PA_\omega$ has all the inference rules of \PA except for $(\forall$R$)$ and \textnormal{(\sf IND)}. Instead, it has the $\omega$-rule
	\begin{prooftree}
		\AxC{$\GD, \varphi(t_1)$}
		\AxC{$\GD, \varphi(t_2)$}
		\AxC{$\dots$}
		\AxC{$\GD, \varphi(t_n)$}
		\AxC{$\dots$}
		\RightLabel{$\omega$}
		\QuinaryInfC{$\GD, \forall v\varphi$}
	\end{prooftree}
\end{definition}

\begin{definition}[$\KFst^\infty$]
	$\KFst^\infty$ extends $\PA_\omega$ with truth-theoretic axioms and truth-theoretic rules of \KFst (formulated in the new language).
\end{definition}

$\KF^\infty$-derivations, due to ($\omega$), are possibly infinite. Notions about derivations introduced above, including height and cut rank, carry over without modifications. In particular, $\KF^\infty$-derivations are well-founded trees, where at each node there is either the root, or instances of axioms, or there is a 1-fold branching (corresponding to unary rules), or a two fold branching (corresponding to binary rules), or an $\omega$-fold branching (corresponding to the $\omega$-rule). Note that if $\KF^\infty_\star\vdash^\alpha\GD, \forall v\varphi$, then $\KF^\infty_\star\vdash^\alpha\Gamma,\varphi(t)$ for any closed term $t$.

Every $\KF^\infty$-derivation can be transformed into a $\KF^\infty$-derivation with height $<\omega^2$ and finite cut rank.\footnote{%
	For sake of readability we omit $\star$-index for the $\KF^\infty$ theories.
} 
It is also well-known that the cost of lowering the cut-rank from $k+1$ to $k$ is exponential with base $\omega$, that is if $\PA_\omega\vdash^\alpha_{k+1}\GD$, then $\PA_\omega\vdash^{\omega^\alpha}_k\GD$. 
It follows that every $\KF^\infty$-derivation of height $\alpha<\varepsilon_0$ and cut rank $m$ can be transformed into a quasi-normal derivation of height $\varphi^m_0\alpha$, where $\varphi^m_0\alpha$ stands for $m$ iterations of the Veblen function $\varphi_0$ on $\alpha$.\footnote{%
	Recall the well-known identities $\varphi_0\alpha=\omega^\alpha$ and $\varphi_10=\varepsilon_0$.
}
We can thus restrict our attention to $\KF^\infty$-derivations of finite cut rank and length $<\varepsilon_0$. These derivations can be primitive recursively encoded by natural numbers, and the codes will contain information about the derivation.\footnote{See \citet[\S4.2.2]{schProof} for details.} In particular, if $u$ codes a derivation \Dmc we can primitive recursively read off from $u$ a bound for the length of \Dmc and a bound for its cut rank; additionally, we can read off the name of: the last inference, its principal/side formulae and its conclusion. This enables us to find a predicate, say $\Bew_\infty(a, \bo r, \cor{\GD})$, expressing the relation $\KF^\infty\vdash^\alpha_\rho\GD$, i.e., $\GD$ is derivable in $\KF^\infty$ with a derivation of length $\leq\alpha$ and cut rank $\leq \rho$. Due to the amount of transfinite induction available in $\PKF^+$ the embedding of $\KF$ in $\KF^\infty$ and partial cut-elimination for $\KF^\infty$ can be formalized in $\PKF^+$.

\begin{lemma}\label{le_Formal}Let 
$\Gamma, \Delta\Sub\lpat$. Then for all $\star$
	\begin{itemize}
		\item[(i)] For all $n,r\in\omega$, $\PKF^{+}_{\star}\vdash\Bew_{\KF_{\star}}(\bo n, \bo r, \cor{\GD}) \then \Bew_\infty(\omega^2, \bo r, \cor{\GD})$.
		\item[(ii)] For $\alpha<\varepsilon_0, \PKF^+_{\star}\vdash\Bew_\infty(a, \bo r, \cor{\GD}) \then \Bew_\infty(\varphi^r_0 a, \bo 0, \cor{\GD})$.
	\end{itemize}	
\end{lemma}

\begin{lemma}\label{le_ILKFSubILPKF+}
	For $\Gamma, \Delta$ sets of literals, for all $\star$, and for all $\alpha<\varepsilon_0$,
\begin{align*}
\PKF^+_\star\vdash \Bew_\infty(a, \bo 0, \cor{\GD})&\text{  implies }\PKF^+_\star\vdash \Gamma^+, \Delta^-\Then\Gamma^+, \Delta^-.
	\end{align*}
\end{lemma}

\begin{proof}
The proof is by transfinite induction on $\alpha<\varepsilon_0$ and to a large extent a formalization of the proof of Lemma~\ref{le_TKFintSubPKF}. For $\alpha=0$ it suffices indeed to formalize the proof of Lemma~\ref{le_TKFintSubPKF}. The case where $\alpha$ is a limit ordinal involves an application of the $\omega$-rule. These cases need not be taken into accout, as they introduce complex formulae. For $\alpha=\beta+1$, the crucial cases are the truth-rules involving quantifiers, as the remaining cases are again immediate by formalizing the proof of Lemma~\ref{le_TKFintSubPKF}. We discuss ($T${\sf all}R)---the other rules can be treated along the same lines. Suppose the $\KF^\infty$-derivation ends with
	\begin{prooftree}
	\def\fCenter{\mbox{ $\GDp$ }}
		\Ax$\vdash^\beta\fCenter, \forall z T(s[\dot{z}/v])$
		\RL{$T${\sf all}R}
		\Un$\vdash^\alpha\fCenter, T(\ud\forall v.s)$
	\end{prooftree}
By inversion on the upper sequent we get, for all closed terms $t$
	\[
		\KF^\infty\vdash^\beta \GDp, T(s[\num(t)/v]).
	\]
This can be formalized within $\PKF^+$, that is, 
	\begin{IEEEeqnarray*}{L}
			\PKF^+\vdash\forall y (\neg\ClTerm(y) \lor \Bew_\infty(b, \bo 0, \cor{\GDp, T(s[\dot{y}/v])}))
		\end{IEEEeqnarray*}
Now, let $u$ be an eigenparameter and recall that $\PA\vdash\ClTerm(\dot{x})$. Then\footnote{%
	Notice that $\PKF^+\vdash\forall v\varphi, \GD, \varphi(t)$ for arbitrary $t$.
}
	\begin{prooftree}
		\AxC{}
		\RL{\PA}
		\UnC{$\Then \ClTerm(\dot{u})$}
		\AxC{$\Then \neg\ClTerm(\dot{u}), \Bew_\infty(b, \bo 0, \cor{\GDp, T(s[\dot{u}/v])})$}
		\RL{Lemma~\ref{le_PKF}(iii)}
		\UnC{$\ClTerm(\dot{u}) \Then \Bew_\infty(b, \bo 0, \cor{\GDp, T(s[\dot{u}/v])})$}
		\RL{Cut}
		\BnC{$\Then \Bew_\infty(b, \bo 0, \cor{\GDp, T(s[\dot{u}/v])})$}
	\end{prooftree}
We apply the induction hypothesis and reason in $\PKF^+_\star$ as follows
	\begin{prooftree}
	\def\fCenter{\mbox{ $\Gamma^+, \Delta'^-\Then\Gamma^-, \Delta'^+$ }}
		\AxC{}
		\RL{i.h.}
		\Deduce$\fCenter, T(s[\dot{u}/v])$
		\RL{$\forall$R}
		\Un$\fCenter, \forall yT(s[\dot{y}/v])$
		\RL{$T${\sf all}R}
		\Un$\fCenter, T(\ud\forall v.s)$
	\end{prooftree}
\end{proof}

Since, as mentioned, $\KF_{\star}$ can be embedded in $\KF^\infty$ without exceeding quasi-normal derivation of length $<\varepsilon_0$, we obtain:

\begin{corollary}\label{lTKF+ILPKF+}
	For $\Gamma, \Delta$ sets of literals, for all $\star$
\begin{align*}
\KF_\star\vdash \GD&\text{  implies }\PKF^+_\star\vdash \Gamma^+, \Delta^-\Then\Gamma^-, \Delta^+.
	\end{align*}
\end{corollary}

This immediately yields that the set of significant inferences of \KF-like theories is contained in the set of provable sequents of appropriate \PKF-like theories.

\begin{lemma}\label{KFSItoPKF}$\KF^\circ\Ssf\Isf\subseteq\PKF^\circ$\end{lemma}
\begin{proof}
	 If $\KF^\circ\vdash\TGD$, then $\PKF^\circ\vdash \TGD$ by Lemma \ref{le_TKFintSubPKF} and Corollaries \ref{TKFrestrPKFrestr}, \ref{lTKF+ILPKF+}. But the truth predicate of $\PKF^\circ$ is transparent and thus $\PKF_\circ\vdash \GD$.
\end{proof}
Lemma \ref{KFSItoPKF} in combination with Corollary \ref{co_ILPKFSubILKF} show that $\PKF^\circ$ yields precisely the significant sentence of $\KF^\circ$. In other words we have answered Question b) of {\sc Reinhardt's Challenge}. 
\begin{proposition}[{\sc Reinhardt's Challenge}]\label{prop_MainIL}
	$\PKF^\circ = \KF^\circ\Ssf\Isf$.
\end{proposition} 
Before we turn our attention to the {\sc Generalized Reinhardt Problem} we remark that an answer to Reinhardt's Question b) yields an answer to Question a) as a corollary, that is, our result subsumes the results provided by \cite{hahAxiomatizing}, \cite{hanCosts}, and \cite{nicProvably}.
\begin{corollary}\label{cor_KFSproof}$\{\varphi\in\Sent_{\lpat}\mid\PKF^\circ\vdash \varphi\}=\KF^\circ\Ssf$.\end{corollary}
Of course, this result also implies that by means of a very simple observation connecting provably true \KF-sequents to provable \PKF-sequents we have reduced questions regarding the proof-theoretic strength of \PKF-like theories to questions regarding the proof-theoretic strength of \KF-like theories.
\begin{corollary}$\KF^\circ$ and $\PKF^\circ$ are proof-theoretically equivalent, i.e.,
\begin{align*}\KF^\circ[\![\PA]\!]\equiv\PKF^\circ[\![\PA]\!].\end{align*}
\end{corollary}

\subsection{Generalized Reinhard Problem}
We promised to give a positive answer to the {\sc Generalized Reinhardt Problem}, that is, the question whether for any pair $\ang{\Gamma, \Delta}\in\KF^\circ\Ssf\Isf$ there is a $\KF^\circ$-derivation such that each node of the derivation tree is a member of $\KF^\circ\Ssf\Isf$. It is now time to make good on our promise. Making essential use of Proposition \ref{prop_MainIL} we show that if $\ang{\Gamma, \Delta}\in\KF^\circ\Ssf\Isf$, then there is a $\KF^\circ$-derivation such that each node of the derivation tree is a $\PKF^\circ$-provable sequent. This implies the desired conclusion in virtue of Proposition \ref{prop_MainIL}.

\begin{proposition}[{\sc Generalized Reinhardt Problem}]\label{prop_ReinKF}
	If $\ang{\Gamma, \Delta}\in\KF^\circ\Ssf\Isf$, then there is a $\KF^\circ$-derivation $\Dmc$ of \GD such that for each node $d$ of \Dmc, $d\in\KF^\circ\Ssf\Isf$.
\end{proposition}

\begin{proof}
	If $\ang{\Gamma, \Delta}\in\KF^\circ\Ssf\Isf$, then by Proposition \ref{prop_MainIL} $\PKF^\circ\vdash\GD$ and hence $\PKF^\circ\vdash \TGD$. Now let $\Dmc'$ be an arbitrary \KFintst-derivation of \TGD, e.g.
	\begin{prooftree}
	\def\fCenter{\mbox{\ $\Then$\ }}
		\AxC{}
		\RightLabel{$\delta_0$}
		\branchDeduce
		\Deduce$\Gamma_3 \fCenter \Delta_3$
		\AxC{}
		\RightLabel{$\delta_1$}
		\DeduceC{$\Gamma_2 \Then \Delta_2$}
		\RL{\Rmc}
		\Bn$\Gamma_1 \fCenter \Delta_1$
		\RightLabel{$\delta_2$}
		\Deduce$T\cor\Gamma \fCenter T\cor\Delta$
	\end{prooftree}
Let $d_0\dots d_{k-1}$ be the nodes of $\Dmc'$. In order to obtain \Dmc, it suffices to replace each $d_i$ with $T\cor\Gamma\ert{d_i}T\cor\Delta$, i.e.:
	\begin{prooftree}
	\def\fCenter{\mbox{\ $\Then$\ }}
		\AxC{}
		\RightLabel{$T\cor\Gamma\ert{\delta_0}T\cor\Delta$}
		\branchDeduce
		\Deduce$T\cor\Gamma \Gamma_3 \fCenter \Delta_3, T\cor\Delta$
		\AxC{}
		\RightLabel{$T\cor\Gamma\ert{\delta_1}T\cor\Delta$}
		\DeduceC{$T\cor\Gamma, \Gamma_2 \Then \Delta_2, T\cor\Delta$}
		\RL{\Rmc}
		\Bn$T\cor\Gamma, \Gamma_1 \fCenter \Delta_1, T\cor\Delta$
		\RightLabel{$T\cor\Gamma\ert{\delta_2}T\cor\Delta$}
		\Deduce$T\cor\Gamma \fCenter T\cor\Delta$
	\end{prooftree}
Since $\KF^\circ$ is closed under weakening, \Dmc is a $\KF^\circ$-derivation of \TGD. But every node $T\cor\Gamma, \Gamma_i \Then \Delta_i, T\cor\Delta$ of \Dmc is derivable in $\PKF^\circ=\KF^\circ\Ssf\Isf$.\end{proof}
It may good to put our answer to the {\sc Generalized Reinhardt Problem} in perspective: we have shown that for every significant inference there is a way to classically derive the sequent such that every node of the derivation is itself a significant inference and hence that every node of the proof is acceptable to the significant reasoner, i.e., the non-classical logician. This does not imply that the non-classical logician can follow the classical reasoning, i.e., that the $\KF^\circ$-proof is also a $\PKF^\circ$-proof. Our result only shows that the classical reasoning is acceptable to the non-classical logician and that \KF-style theories can be used instrumentally. It does not show that one can always reason non-classically in $\KF^\circ$. But if the latter were the case, it seems that $\KF^\circ$ would deliver an independent axiomatization of its significant part in its own right. Surely---while it is certainly an interesting question whether for every $\KF^\circ$-significant inference there is a $\KF^\circ$-derivation, which is also a $\PKF^\circ$-derivation---, such a result is not required for an instrumental interpretation of $\KF^\circ$ and left for future research.


\section{Conclusion}
\label{sec:ConcludingRemarks}
In this paper we had a fresh look at Reinhardt's program and proposed  to focus on the provably true inferences of \KF-like theories rather than the provably true sentences only. We showed that if we conceive of the significant part of \KF-like theories as the set of provably true inferences then we can remain within the significant part of the theory in proving its significant inferences. This answers the {\sc Generalized Reinhardt Problem} and also shows that we need not step outside the significant part of \KF in proving theorems of the form $T\cor\varphi$, which was the content of the original formulation of {\sc Reinhardt's Problem}. From the perspective of \KF-derivations the use of the nonsignificant part of \KF is hence dispensable and in this sense an instrumentalist interpretation of \KF is certainly available. However, should we conclude that we have justified the use of nonsignificant sentences entirely within the framework of significant sentences?

One may think that to answer the latter question affirmatively an independent characterization of the significant part of \KF needs to be provided and this is precisely the content of {\sc Reinhardt's challenge}. Building on results by \cite{hahAxiomatizing,hanCosts} and \cite{nicProvably} we have shown how to provide axiomatizations of the significant part of \KF-like theories in non-classical logic. The only remaining question is whether these axiomatizations are fully independent. We take it that there is no doubt in this respect concerning the axiomatizations of the significant part of \KF-like theories with internal or restricted induction. Turning to \KF-like theories with full induction the crucial question is whether the rule ($\mathsf{TI}^{<\varepsilon_0}$) is available from within the significant framework. Ultimately, an answer to this question will depend on the role the theory of truth is supposed to play within one's theoretical framework. If, for instance, one takes the theory to play an important role in the foundations of mathematics and, for instance, to play a role in singling out the limits of predicativity \citep[cf.][]{fefReflecting}, then one should arguably refrain from thinking that ($\mathsf{TI}^{<\varepsilon_0}$) can be assumed without further justification from within the significant framework. But, to the contrary, if the theory of truth is to play no role in the foundations of mathematics and classical mathematical theorizing is freely available from within the significant framework, then it is hard to see why ($\mathsf{TI}^{<\varepsilon_0}$) should not be considered as fully justified from within the significant perspective. In this case it would seem that Reinhardt's program needs to be deemed successful. However, a proper philosophical assessment of the rule of transfinite induction up to $\varepsilon_0$ from the perspective of Reinhardt's program is beyond the scope of the paper. Nonetheless, as we hope to have established in this paper---contra \cite{hahAxiomatizing} and \cite{hanCosts}---there is no major technical obstacle preventing the success of Reinhardt's program and, in this sense, Reinhardt was certainly right in claiming that {\it ``the chances of success in this context (\dots) are somewhat better than in Hilbert's context''} \citep[p.225]{reiRemarks}.

\subsection*{Funding acknowledgement}Luca Castaldo's research was supported by the AHRC South, West and Wales Doctoral Training Partnership (SWW DTP). Johannes Stern's research is funded by the ERC Starting Grant TRUST 803684. We thank Carlo Nicolai for very helpful comments and, in particular, for pointing to a substantial simplification of our proof of Proposition \ref{prop_MainIL}.


\bibliographystyle{plainnat}
\bibliography{bib}{}


\end{document}